\documentclass[11pt]{article}

\usepackage{amsmath,amsfonts,amsthm,amscd,amssymb,graphicx}
\numberwithin{equation}{section}

\usepackage{color}

\newtheorem{theorem}{Theorem}[section]

\newtheorem{lemma}[theorem]{Lemma}

\def\D{\partial}
\def\dt{\partial_t}

\def\ep{\epsilon}
\def\oD{\mathrm{D}}

\def\R{\Re e}
\def\I{\Im m}

\def\I{\Im m}

\newcommand{\RR}{\mathbb{R}}
\newcommand{\DD}{{\mathbb D}}

\def\cV{\mathcal{V}}
\def\cH{\mathcal{H}}

\def\vj{{\bf j}}

\def\vA{{\bf A}}
\def\vE{{\bf E}}
\def\vB{{\bf B}}

\begin{document}

\title{Stability analysis of collisionless plasmas with specularly reflecting boundary}

\author{Toan Nguyen\footnotemark[1] \and Walter A. Strauss\footnotemark[2]}

\date{\today}

\maketitle

\begin{abstract} 
In this paper we provide sharp criteria for linear stability or instability of equilibria of 
collisionless plasmas in the presence of boundaries. 
Specifically, we consider the relativistic Vlasov-Maxwell system with specular reflection at the 
boundary for the particles and with the perfectly conducting boundary condition for the electromagnetic field.  
Here we initiate our investigation in the simple geometry of radial and longitudinal symmetry.  
\end{abstract}

\renewcommand{\thefootnote}{\fnsymbol{footnote}}

\footnotetext[1]{Division of Applied Mathematics, Brown University, 182 George Street, Providence, RI 02912, USA. 
Email: Toan\underline{~}Nguyen@Brown.edu. Research of T.N. was partially supported under NSF grant no. DMS-1108821.
}

\footnotetext[2]{Department of Mathematics and Lefschetz Center for Dynamical Systems,
Brown University, Providence, RI 02912, USA. Email: wstrauss@math.brown.edu.}




\section{Introduction}

We consider a plasma at high temperature or of low density such that collisions can be ignored as compared 
with the electromagnetic forces.  Such a plasma is modeled by the relativistic Vlasov-Maxwell system (RVM) 
\begin{equation}\label{Vlasov-equations}\left\{\begin{aligned} 
&\dt f^+ + \hat v \cdot \nabla_x f^+ + (\vE + \hat v \times \vB )\cdot \nabla_v f^+   =0,
\\&\dt f^- + \hat v \cdot \nabla_x f^- - (\vE + \hat v \times \vB ) \cdot \nabla_v f^-  =0, 
\end{aligned}\right.
\end{equation}
\begin{equation}\label{Gauss-laws} 
\nabla_x \cdot \vE = \rho,\qquad \nabla_x \cdot \vB =0,
\end{equation}
\begin{equation}\label{Ampere-law}
\dt  \vE  - \nabla_x \times \vB  = - \vj, \quad 
\dt \vB + \nabla_x \times \vE =0,
\end{equation}
$$\rho = \int_{\RR^3} (f^+ - f^-) \;dv , \qquad \vj = \int_{\RR^3} \hat v (f^+-f^-)\; dv.$$ 
Here 
$f^\pm(t,x,v)\ge0$ is the density distribution for ions and electrons, respectively,  $x \in \Omega \subset \RR^3$ 
is the particle position, $\Omega$ is the region occupied by the plasma,  $v$ is the particle momentum, 
$\langle v \rangle = \sqrt{1+|v|^2}$ is the particle energy, 
$\hat v = v/\langle v \rangle$ the particle velocity, $\rho$  the charge density, $\vj$  the current density, 
$\vE$  the electric field, $\vB$  the magnetic field and $\pm(\vE + \hat v \times \vB)$ the electromagnetic force.   
We assume that the particle molecules interact with each other only through their own electromagnetic forces.
For simplicity, we have taken all physical constants such as the speed of light and the mass of the electrons and ions 
equal to 1.  This whole paper can be easily modified to apply with the true physical constants.  

Stability analysis for a Vlasov-Maxwell system of the type that we present in this paper has so far appeared 
only in the absence of spatial boundaries, that is, 
either in all space or in a periodic setting like the torus.  In this paper we present the first systematic stability analysis 
in a domain $\Omega$ with a boundary.  It is an unresolved problem to determine which boundary conditions 
an actual plasma may satisfy under various physical conditions.  
Several  boundary conditions are mathematically valid and some of them are more physically justified than others.  
Stability analysis is a central issue in the theory of plasmas.  
In a tokamak and other nuclear fusion reactors, for instance, the plasma is confined by a strong magnetic field.  
This paper is a first, rather primitive, step in the direction of mathematically understanding a confined plasma.  
We take the case of a fixed boundary with specular and perfect conductor boundary conditions 
in a  longitudinal and radial setting.    

The specular condition is 
\begin{equation}\label{bdry-specular} f^\pm (t,x,v) = f^\pm(t,x,v - 2(v\cdot n(x))n(x)),\qquad n(x)\cdot v <0,\qquad x \in \D \Omega, \end{equation}
where $n(x)$ denotes the outward normal vector of $\D \Omega$ at $x$. 
The perfect conductor boundary condition is 
\begin{equation}\label{bdry-EBcond} 
\vE\times n(x) = 0, \qquad \vB \cdot n(x)  =0 ,\qquad x \in \D \Omega.\end{equation}
Under these conditions it is straightforward to see that the total energy  
\begin{equation}\label{energy-non} 
\mathcal{E}(t)  = \frac{1}{2}\int_\DD \int_{\RR^3}v^2 (f^+ + f^-)\; dvdx + \frac 12 \int_\DD \Big( |\vE|^2 + |\vB|^2\Big)\; dx \end{equation}
is conserved in time, and also that the system admits infinitely many equilibria.  {\em The main focus of the present paper is to investigate stability properties of the equilibria.} 

Our analysis closely follows the spectral analysis approach in \cite{LS1,LS2,LS3} which tackled the stability problem 
in domains without spatial boundaries.  Roughly speaking, that approach provided the sharp stability  
criterion $\mathcal{L}^0\ge 0$, where $\mathcal{L}^0$ is a certain nonlocal self-adjoint operator 
acting on scalar functions that depend only on the spatial variables.  
The positivity condition was verified explicitly for various interesting examples. 
It may also be amenable to numerical verification.   
In our case with a boundary, every integration by parts brings in boundary terms.  
This leads to some significant complications.

In the present paper, we restrict ourself to the stability problem in the simple setting of 
longitudinal and radial symmetry.    Thus the problem becomes spatially two-dimensional.  
Indeed, using standard cylindrical coordinates $(r,\theta, z)$, the symmetry means that 
 there is no dependence on $z$ and $\theta$ and that 
the domain is a cylinder $\Omega = \DD\times \RR$ where $\DD$ is a disk.  
We may as well assume that $\DD$ is the unit disk in the $(x_1,x_2)$ plane.  
It follows that 
$x = (x_1,x_2,0)\in \DD\times \RR$, $v = (v_1,v_2,0)\in \RR^3$,  $\vE = (E_1,E_2,0)$, and $\vB = (0,0,B)$.  
In the sequel we will drop the zero coordinates so that $x\in\DD$ and $v\in\RR^2$.  
In terms of the polar coordinates $(r,\theta)$, we denote  
$ e_r = (\cos \theta,\sin \theta), \ e_\theta = (-\sin \theta,\cos \theta).$ 
It follows that the field has the form 
\begin{equation}\label{def-EB}
\vE = -\partial_r \varphi e_r  -\dt \psi e_\theta, \qquad B =  \frac 1r \partial_r (r\psi)),
\end{equation}
where the scalar potentials $\varphi(t,r)$ and $\psi(t,r)$ satisfy a reduced form of the Maxwell equations.  
See the next section for details.  

\subsection{Equilibria}\label{sec-equilibrium}
We will denote an equilibrium by $(f^{0,\pm}, \vE^0, B^0)$.  Its  field has the form 
\begin{equation}\label{EB-equil}
\vE^0 = -\partial_r \varphi^0 e_r, \qquad B^0 = \frac 1r \partial_r (r \psi^0) .
\end{equation}
Then the particle energy and angular momentum  
\begin{equation} \label{ep}
e^\pm (x,v):= \langle v \rangle  \pm \varphi^0(r), \qquad p^\pm(x,v) := r(v_\theta \pm \psi^0(r)),  
\end{equation} 
are invariant along the particle trajectories.  
It is straightforward to check that $\mu^\pm(e^\pm,p^\pm)$ solve the Vlasov equations for any smooth functions 
$\mu^\pm(e,p)$. So we consider equilibria of the form  
\begin{equation}\label{f-equil} 
f^{0,+}(x,v) = \mu^+ (e^+(x,v),p^+(x,v)),\quad f^{0,-}(x,v) = \mu^- (e^-(x,v),p^-(x,v)).  
\end{equation}
The potentials still have to satisfy the Maxwell equations, which take the form 
 \begin{equation}\label{eqs-equilibria} \begin{aligned}
-\Delta \varphi^0& =\int  \Big[\mu^+(e^+,p^+) - \mu^-(e^-,p^-) \Big]\;dv
 \\- \Delta_r\psi^0 &=  \int  \hat v_\theta\Big[\mu^+(e^+,p^+) - \mu^-(e^-,p^-) \Big]\;dv
\end{aligned}
\end{equation}
with $\Delta_r = \Delta - \frac1{r^2}$.   Again, see the next section for details. 
It is clear that {\it the boundary conditions \eqref{bdry-specular} and \eqref{bdry-EBcond} 
are automatically satisfied for the equilibria} since $e^\pm$ and $p^\pm$ are even in $v_r$, and $\vE^0$ is parallel to $e_r$.  
In the appendix we will prove that plenty of such equilibria do exist.   

Let $(f^{0,\pm},\vE^0,\vB^0)$ be an equilibrium as just described with $f^{0,\pm} = \mu^\pm(e^\pm,p^\pm)$. 
We assume that $\mu^\pm(e,p)$ are nonnegative, $C^1$ smooth, and satisfy 
\begin{equation}\label{mu-cond} 
\mu^\pm_e (e,p)<0 ,\qquad |\mu_p^\pm(e,p)| + |\mu_e^\pm(e,p)| + \frac{|\mu^\pm_p(e,p)|^2}{|\mu_e^\pm(e,p)|}\quad\le\quad \frac{ C_\mu}{1+|e|^\gamma}
\end{equation}
for some constant $C_\mu$ and some $\gamma>2$, where the subscripts $e$ and $p$ denote the partial derivatives.   
In addition, we also assume that $\varphi^0,\psi^0$ are continuous in $\overline \DD$.   
It follows that $\vE^0,\vB^0\in C^1(\overline\DD)$, as proven in the appendix.  



We consider the Vlasov-Maxwell system linearized around the equilibrium. The linearization is 
\begin{equation}  \label{linearized-sys} 
 \begin{aligned}
 \dt f^\pm+\oD ^\pm f^\pm  &=  \mp (\vE + \hat v\times \vB)\cdot \nabla_v f^{0,\pm} ,
\end{aligned} 
\end{equation}
together with the Maxwell equations and the specular and perfect conductor boundary conditions.  
Here $\oD ^\pm$ denotes the first-order linear differential operator: 
$\oD^\pm: = \hat v \cdot \nabla_x  \pm (\vE^0 + \hat v\times \vB^0)\cdot \nabla_v. $ 
See the next section for details.  

\subsection{Spaces and operators} \label{sec-spaces}
In order to state precise results, we have to define certain spaces and operators.    
We denote by $\cH^\pm = L^2_{|\mu_e^\pm|}(\DD \times \RR^2)$ the weighted $L^2$ space consisting of functions 
$f^\pm(x,v)$ which are radially symmetric in $x$ such that 
$$ \int_\DD  \int |\mu_e^\pm| |f^\pm|^2 \; dv dx <+\infty. $$ 
The main purpose of the weight function is to control the growth of $f^\pm$ as $|v| \to\infty$.  
Note that 
the weight $|\mu_e^\pm|$ never vanishes and it decays like a power of $v$ as $|v| \to\infty$.  
When there is no danger of confusion, we will often write $\cH=\cH^\pm$.  

For $k\ge 0$  
we denote by $H^k_r(\DD )$ the usual $H^k$ space on $\DD$ 
that consists of functions that are radially symmetric. 
If $k=0$ we write   $L^2_r(\DD )$.  
By  $H^{k\dagger}(\DD  )$ we denote the space  of functions $\psi = \psi(r)$ in $H^k_r(\DD )$ such that 
 $\psi(r)e^{i\theta}$ belongs to the usual $H^k(\DD  )$ space. The motivation for this  space is to 
 get rid of the apparent singularity $1/{r^2}$ at the origin in the operator $\Delta_r$, thanks to the identity 
$$ - \Delta_r \psi = \Big(-\Delta + \frac 1{r^2}\Big )\psi = - e^{-i \theta} \Delta (\psi e^{i\theta}).$$
By $\mathcal{V}$ we denote the space consisting of the functions in $H^2_r(\DD  )$ which satisfy 
the Neumann boundary condition and which have zero average over $\DD$.   
Also, let $\mathcal{V}^\dagger$ be the space consisting of the functions in $H^{2\dagger}(\DD )$ which satisfy 
the Dirichlet boundary condition. The spaces $\mathcal{V}$ and $\mathcal{V}^\dagger$ incorporate the 
boundary conditions  \eqref{BC-Maxwell} for electric and magnetic potentials, respectively.

Denote by $\mathcal{P}^\pm$ the orthogonal projection on the kernel of $\oD ^\pm$ 
in the weighted  space $\cH^\pm$. In the  spirit of \cite{LS1,LS3},  our main results involve the three linear operators
on $L^2(\DD)$, two of which are unbounded, 
\begin{equation}\label{operators-0}
\begin{aligned}  \mathcal{A}_1^0 h & = - \Delta h  - \int \mu^+_e(1-\mathcal{P}^+) h \; dv - \int \mu^-_e (1-\mathcal{P}^-) h \; dv, 
\\
 \mathcal{A}_2^0 h & =- \Delta_r h   -  \int r\hat v_\theta (\mu_p^+ + \mu_p^-) \; dv h  
 -   \int \hat v_\theta\Big(\mu_e^+ \mathcal{P}^+(\hat v_\theta h) +\mu_e^- \mathcal{P}^-(\hat v_\theta h) \Big)\; dv, 
\\
\mathcal{B}^0 h  &= r \int \Big(\mu_p^+ +  \mu_p^-\Big )\; dvh + 
\int \Big(\mu_e^+\mathcal{P}^+(\hat v_\theta h) +\mu_e^-\mathcal{P}^-(\hat v_\theta h) \Big) \; dv . 
\end{aligned}\end{equation}
Here $\mu^\pm$ denotes $\mu^\pm(e^\pm,p^\pm) = \mu^\pm(\langle v \rangle  \pm \varphi^0(x),rv_\theta \pm r\psi^0(x))$. 
These three operators are naturally derived from the Maxwell equations when $f^+$ and $f^-$ are written in integral 
form by integrating the Vlasov equations along the trajectories.  See Section \ref{ss-keyops} for their properties.  
In the next section 
we will show that both $\mathcal{A}_1^0$ with domain  $\mathcal{V}$ and $\mathcal{A}_2^0$ 
with domain  $\mathcal{V}^\dagger$ are self-adjoint operators on $L^2_r(\DD)$.   
Furthermore, the inverse of $\mathcal{A}_1^0$ is well-defined on the range of $\mathcal{B}^0$, and so we are able 
to introduce our key operator 
\begin{equation}\label{operator-L0}
\mathcal{L}^0 = \mathcal{A}_2^0 + (\mathcal{B}^0)^* (\mathcal{A}_1^0)^{-1} \mathcal{B}^0. \end{equation}
The operator $\mathcal{L}^0$ will then be self-adjoint on $L^2_r(\DD)$ with its domain $\mathcal{V}^\dagger$.   
As the next theorem states, $\mathcal{L}^0 \ge 0$ [which means 
that $(\mathcal{L}^0\psi,\psi)_{L^2} \ge 0$ for all $\psi \in \mathcal{V}^\dagger$]  is the condition for stability. 

Finally, by a {\it growing mode} we mean a solution of the linearized system (including the 
boundary conditions) of the form 
$(e^{\lambda t} f^\pm, e^{\lambda t}{\vE}, e^{\lambda t} {\vB})$ with $\R\lambda>0$ 
such that $f^\pm \in \cH^\pm$ and $\vE, \vB\in L^2(\DD)$.  
The derivatives and the boundary conditions are considered in the weak sense, which will be 
 justified in Lemma \ref{growprops}.  
In particular, the weak meaning of the specular condition on $f^\pm$ will be given by \eqref{domainD}.

\bigskip  
\subsection{Main results}
The first main result of our paper gives a necessary and sufficient condition for stability in the spectral sense.   
\begin{theorem}\label{theo-main} 
Let $(f^{0,\pm},\vE^0,\vB^0)$ be an equilibrium of the Vlasov-Maxwell system satisfying \eqref{mu-cond}. 
Consider the linearization of the Vlasov-Maxwell system \eqref{linearized-sys} for {\em radially symmetric} perturbations 
together with the specular and perfect conductor boundary conditions. Then

(i) if  $\mathcal{L}^0\ge 0$, there exists no growing mode of the linearized system;   

(ii) any growing mode, if it does exist, must be purely growing; that is, the unstable exponent $\lambda$ must be real; 

(iii) if $\mathcal{L}^0\not\ge 0$, there exists a growing mode.

\end{theorem}

Our second main result provides explicit examples for which the stability condition does or does not hold. 
For more precise statements of this result, see Section \ref{sec-examples}.

\begin{theorem} \label{theo-examples} Let $(\mu^{\pm},\vE^0,\vB^0)$ be an equilibrium as above.   

(i) The condition $p \mu^\pm_p(e,p)\le 0$ for all $(e,p)$ implies $\mathcal{L}^0 \ge 0$, 
provided that $\varphi^0$ is bounded and $\psi^0$ is sufficiently small.     (So such an equilibrium is stable.) 

(ii) The condition $|\mu_p^\pm (e,p)|  \le  \frac \epsilon {1+|e|^\gamma} $ for some $\gamma>2$ and for $\epsilon$ sufficiently small implies $\mathcal{L}^0 \ge 0$, 
provided that $\varphi^0=0$.  Here $\psi^0$  is not necessarily small.  (So such an equilibrium is stable.) 
 
(iii) The conditions $\mu^+(e,p) = \mu^-(e,-p)$ and $p \mu^-_p(e,p)\ge c_0 p^2 \nu(e)$, for some 
nontrivial nonnegative function $\nu(e)$, imply that for a suitably scaled version of  $(\mu^\pm,0,\vB^0)$,  
 $\mathcal{L}^0\ge 0$ is violated.  (So such an equilibrium is unstable.)  
\end{theorem}

The instabilities in a plasma are due to the collective behavior of all the particles.  For a homogeneous equilibrium 
(without an electromagnetic field) Penrose \cite{Pen}  found a beautiful necessary and sufficient condition for stability 
of the Vlasov-Poisson system (VP).  For a BGK mode, the equilibrium has an electric field.  
In that case, proofs of instability for VP (electric perturbations) were first given in \cite{GS95a, GS95b, GS98}
and especially in \cite{Lin, Lin2} for non-perturbative electric fields.  
Once magnetic effects are included, even for a homogeneous plasma, the situation becomes much 
more complicated: see for instance \cite{GuoCPAM97, GS99, GS2000}.  
In a series of three papers \cite{LS1,LS2,LS3} a more general approach was taken that treated fully electromagnetic equilibria 
with electromagnetic perturbations.  The linear stability theory was addressed in \cite{LS1,LS3}, while 
in \cite{LS2} fully nonlinear stability and instability was proven in some cases.  
In all of the work just mentioned, boundary behavior was not addressed; the spatial domains were 
either all space or periodic.  

In this paper we do not address the question of well-posedness of the initial-value problem.  
For VP in $\RR^3$ well-posedness was 
proven in \cite{Pfaff} and \cite{PerthLions}.  For RVM in all space $\RR^3$ it is a famous open problem.  
The relativistic setting seems to be required, for otherwise the Vlasov and Maxwell characteristics would collide.     
However, for RVM in the whole plane $\RR^2$, which is the case most relevant to this paper, 
well-posedness and regularity were proven in \cite{GSch2D}.  
The same is true even in the 2.5 dimensional case  \cite{GSch2.5D}.  
Although global weak solutions exist in $\RR^3$, they are not known to be unique  \cite{DiL}.  
Furthermore, in a spatial domain with a boundary on which one assumes specular and perfect conductor conditions, 
global weak solutions also exist  \cite{Guo1993}.  
Well-posedness and regularity of VP in a convex bounded domain with the specular condition 
was recently proven in \cite{Hwang-Vel}.

A delicate part of our analysis is how to deal with the specular boundary condition within the context of weak solutions.  
This is discussed in subsection \ref{ss-Vlasov}. Properly formulated, the operators $D^\pm$ then are skew-adjoint.  
In this paper, as distinguished from \cite{LS1}, we entirely deal with the weak formulation.  
Our regularity assumptions are essentially optimal.  
In subsection \ref{ss-growing} we prove that the densities $f^\pm$ of a growing mode of the linearized system 
decay at a certain rate as $|v|\to\infty$.  

As in \cite{LS1}, the stability part of Theorem \ref{theo-main} is based on realizing the temporal invariants
of the linearized system.  They have to be delicately calculated due to the weak form of the boundary conditions.  
This is done in subsection \ref{inv}.  The invariants are  the generalized energy $\mathcal I$ and the 
casimirs $\mathcal K_g$, which are a consequence of the assumed symmetry of the system.  
A key part of the stability proof is to minimize the energy with the magnetic potential being held fixed (see subsection \ref{ss-min}).  
The purity of any growing mode, in part {\it (ii)} of Theorem \ref{theo-main}, is a consequence of 
splitting the densities into even and odd parts relative to the variable $v_r$ (see subsection \ref{ss-pure}).   
The proof of stability is completed in subsection \ref{ss-stab}.  

The proof of instability in Theorem \ref{theo-main}{\it(iii)} requires the introduction of a family of 
linear operators $\mathcal{L}^\lambda$ which formally reduce to $\mathcal L^0$ as $\lambda\to0$.  
The technique was first introduced by Lin in \cite{Lin} for the BGK modes.  
These operators explicitly use the particle paths (trajectories) $(X^\pm(s;x,v), V^\pm(s;x,v))$.  
{\it The trajectories reflect specularly a countable number of times at the boundary.}  
We use them to represent the densities $f^\pm$ in integral form, like a Duhamel representation.  
This representation together with the Maxwell equations leads to the family of operators in subsection \ref{ss-operators}.
Self-adjointness requires careful consideration of the trajectories.  
It is then shown in subsection \ref{ss-constr} that $\mathcal{L}^\lambda$ is a positive operator for large $\lambda$, while 
it has a negative eigenvalue for $\lambda$ small because of the hypothesis $\mathcal L^0 \not\ge0$.  
Therefore $\mathcal{L}^\lambda$ has a nontrivial kernel for some $\lambda>0$
If $\psi$ is in the kernel, it is the magnetic potential of the growing mode.  From $\psi$ we construct 
the corresponding electric potential $\varphi$ and densities $f^\pm$.  

In Section \ref{sec-examples} we prove Theorem \ref{theo-examples}.  The stability examples are relatively easy.  
As we see, the basic stabilizing condition  is $p\mu^\pm_p \le0$.  
To construct the unstable examples we make the simplifying assumption that the equilibrium has no electric field 
so that $\mathcal L^0 = \mathcal A_2^0$.  In the expression $(\mathcal{L}^0\psi,\psi)_{L^2}$ the term that has to 
dominate negatively is the one with  $p\mu_p$.  In order to make it dominate, we scale the equilibrium appropriately.  
We first treat the homogeneous case (Theorem \ref{lem-mg-unstab}) and then the purely magnetic case 
(Theorem \ref{lem-unstab-inhomo}).

%
%
%
%


%
%
%
%

\section{The symmetric system}


\subsection{The potentials} \label{ss-pot} 
It is convenient when dealing with the Maxwell equations to introduce the electric scalar potential $\varphi$ and magnetic vector potential $\vA$ through  
\begin{equation}\label{def-potentials} \vE = -\nabla \varphi - \D_t \vA, \qquad  \vB = \nabla \times \vA , \end{equation}
in which without loss of generality, we impose the Coulomb gauge constraint $\nabla \cdot \vA =0$. Note that with these forms, there automatically hold the two Maxwell equations:
$$\dt \vB + \nabla \times \vE =0 ,\qquad \nabla \cdot \vB = 0,$$
whereas the remaining two Maxwell equations become
$$ - \Delta \varphi = \rho ,\qquad \quad \partial_t^2 \vA  - \Delta \vA + \partial_t \nabla \varphi  = \vj.$$

Under the assumption of  radial and longitudinal symmetry, there is no $z$ or $\theta$ dependence.   
We use the polar coordinates $x = (r \cos \theta ,r \sin\theta)$ on the unit disk $\DD$.   
A radial function $f(x)$ is one that depends only on $r$ and in this case we often abuse  notation by writing $f(r)$.   
We also denote the unit vectors by $ e_r = (\cos \theta,\sin \theta)$ and $e_\theta = (-\sin \theta,\cos \theta)$.  
Thus $e_r (x)= n(x)$ is the outward normal vector at $x\in\D\DD$.  
Although the functions do not depend on $\theta$, the unit vectors $e_r$ and $e_\theta$ do.  
Then we may write 
$f^\pm = f^\pm(t,r,v)$, where $v = v_re_r + v_\theta e_\theta$ and $\vA = A_re_r+A_\theta e_\theta$. 

Now the Coulomb gauge in this symmetric setting reduces to $\frac 1r \partial_r(r A_r) =0$,  
so that $A_r = h(t)/r$.  We require the field to have finite energy, meaning that $\vE,\vB \in L^2(\DD)$.  
Thus $E_r = -\D_r \varphi - \D_t A_r \in  L^2(\DD)$ only if $h(t)$ is a constant.  
But if $h(t)$ is a constant, we may as well choose it to be zero because it will not contribute to either $\vE$ or $\vB$.  
For notational convenience let us write $\psi$ in place of $A_\theta$. 
The fields defined through \eqref{def-potentials} then take the form 
\begin{equation}\label{def-EB}
\vE = -\partial_r \varphi e_r  -\dt \psi e_\theta, \qquad \vB = B e_z, \qquad B =  \frac 1r \partial_r (r\psi)).
\end{equation}
We note that 
$$
\hat v\cdot\nabla_x f  =  \hat v_r \D_rf + \frac 1r \hat v_\theta \D_\theta f  
=  \hat v_r \D_rf + \frac 1r \hat v_\theta \Big\{v_\theta \D_{v_r} f  -  v_r \D_{v_\theta} f \Big\} . $$
In these coordinates the RVM system takes the form  
\begin{equation}\label{VMsys}
\left\{ \begin{array} {rlll} 
\dt f^+ + \hat v_r \partial_r f^+ + \Big(E_r + \hat v_\theta B+\frac 1r  v_\theta \hat v_\theta\Big)\D_{v_r}f^+ + \Big(E_\theta - \hat v_r B -\frac 1r  v_r \hat v_\theta\Big) \D_{v_\theta} f ^+=0,
\\
\dt f^- + \hat v_r \partial_r f^-  -\Big(E_r + \hat v_\theta B - \frac 1r v_\theta \hat v_\theta \Big) \D_{v_r}f^- -  \Big(E_\theta - \hat v_r B  + \frac 1r  v_r \hat v_\theta\Big) \D_{v_\theta} f ^-=0,
\end{array} \right.  
\end{equation}
\begin{equation}\label{Maxwell-eqs} 
 \begin{aligned}   
        - \Delta \varphi & =\rho =\int_{\RR^2} (f^+ - f^-)(t,r,v) \;dv, \\
 \dt \partial_r \varphi &= j_r = \int_{\RR^2} \hat v_r (f^+ - f^-)(t,r,v)\; dv,  \\
 (\partial_t^2 - \Delta_r) \psi &=  j_\theta = \int_{\RR^2} \hat v_\theta (f^+ - f^-)(t,r,v)\; dv,  
\end{aligned}   
 \end{equation}
where $\Delta_r = \Delta - \frac 1{r^2}$.    
The system \eqref{VMsys} - \eqref{Maxwell-eqs} is accompanied by the specular boundary condition on $f^\pm$,  
which is now equivalent to the evenness of $f^\pm$ in $v_r$ at $r=1$.  In particular, $j_r=0$ on $\D\DD$.    
The condition $\vB\cdot n =0$ is automatic.  
The  boundary conditions on $\varphi$ and $\psi$  are 
\begin{equation}\label{BCs} \partial_r\varphi (t,1) = const.,\qquad \psi (t,1)= const.\end{equation}
The Neumann condition on $\varphi$ comes naturally from the second Maxwell equation in \eqref{Maxwell-eqs} 
with $j_r=0$.  
 The Dirichlet condition on $\psi$ comes naturally from  $0 = \vE \times n = (0,0,-\partial_t \psi) $. 

%

\subsection{Linearization} \label{ss-lin}
We linearize the Vlasov-Maxwell near the equilibrium $(f^{0,\pm},\vE^0,B^0)$.  From \eqref{VMsys}, the linearized 
Vlasov equations are 
\begin{equation}  \label{linearization} 
 \begin{aligned}
 \dt f^++\oD ^+f^+  &= - (E_r + \hat v_\theta B)\D_{v_r}  f^{0,+}- (E_\theta - \hat v_r B ) \D_{v_\theta}  f^{0,+} ,
\\ \dt f^-+\oD ^-f^-  &=  ~~(E_r + \hat v_\theta B)\D_{v_r}  f^{0,-}+ (E_\theta - \hat v_r B ) \D_{v_\theta}  f^{0,-}  .
\end{aligned}                        \end{equation}
The first-order differential operators 
$\oD^\pm: = \hat v \cdot \nabla_x  \pm (\vE^0 + \hat v\times \vB^0)\cdot \nabla_v $ 
now take the form 
\begin{equation}\label{def-opD}
\begin{aligned}
\oD ^+:= \hat v_r \D_r  + \Big(E_r^0 + \hat v_\theta B^0+\frac 1r v_\theta \hat v_\theta\Big)\D_{v_r}  - \Big(\hat v_r B^0 +\frac 1r v_r \hat v_\theta\Big)   \D_{v_\theta}  ,
\\
\oD ^-:= \hat v_r \D_r   - \Big(E_r^0 + \hat v_\theta B^0- \frac 1r v_\theta \hat v_\theta\Big)\D_{v_r}  + \Big(\hat v_r B^0  - \frac 1r v_r \hat v_\theta\Big)   \D_{v_\theta}  .
\end{aligned}
\end{equation}
In order to compute the right-hand sides of \eqref{linearization} more explicitly, 
we differentiate the definition $f^{0,\pm} = \mu^\pm(e^\pm,p^\pm)$ to get 
$$ \D_{v_r}  f^{0,\pm} = \mu_e^\pm \hat v_r, \qquad \D_{v_\theta}  f^{0,\pm} = \mu_e^\pm \hat v_\theta + r \mu_p^\pm .$$
Thus, together with the forms of $\vE$ and $B$  in \eqref{def-EB}, we  calculate 
$$\begin{aligned}
 (E_r + \hat v_\theta B)&\D_{v_r}  f^{0,+}+ (E_\theta - \hat v_r B ) \D_{v_\theta}  f^{0,+}     \\
 & =  (E_r + \hat v_\theta B) \mu^+_e \hat v_r + (E_\theta - \hat v_r B ) (\mu^+_e \hat v_\theta + r \mu^+_p)   \\
 & = -\mu^+_e \hat v_r \D_r \varphi - \mu^+_p \hat v_r \D_r (r\psi) - \mu^+_e \hat v_\theta \dt \psi - r \mu^+_p \dt \psi   \\
 &= -\mu^+_e\oD ^+ \varphi  - \mu^+_p\oD ^+(r\psi)  - ( \mu^+_e \hat v_\theta  + r\mu^+_p)  \dt \psi,
 \end{aligned}$$
where the last line is due to the fact that $\oD ^+ \varphi = \hat v_r \D_r \varphi$ for radial functions.  
Of course a similar calculation holds for $f^{0,-}$.   
Thus the linearization \eqref{linearization} becomes
\begin{equation}\label{lin-VM} \begin{aligned}
 \dt f^++\oD ^+f^+  
 &=~~ \mu^+_e\oD ^+ \varphi  + \mu^+_p\oD ^+(r\psi)  +  \D_{v_\theta}[\mu^+] \dt \psi \\  
 &=  \ \   \mu_e^+\hat v_r \D_r\varphi  +  \mu_p^+ \hat v_r \D_r(r\psi) + ( \mu^+_e \hat v_\theta  + r\mu^+_p)  \dt \psi \\ 
 \dt f^-+\oD ^-f^- 
&=  - \mu^-_e\oD ^- \varphi  - \mu^-_p\oD ^-(r\psi)  -  \D_{v_\theta}[\mu^-]\dt \psi \\  
&=   - \mu_e^-\hat v_r \D_r\varphi  -  \mu_p^- \hat v_r \D_r(r\psi)  - ( \mu^-_e \hat v_\theta  + r\mu^-_p)  \dt \psi 
\end{aligned}
\end{equation}
Of course,  linearization does not alter the Maxwell equations \eqref{Maxwell-eqs}.   
As for boundary conditions, we naturally take the specular condition on $f^\pm$ and 
\begin{equation}\label{BC-Maxwell} 
\D_r \varphi(t,1) = 0, \qquad \psi(t,1) =0.
\end{equation}

\subsection{The Vlasov operators}\label{ss-Vlasov}

The Vlasov operators $\oD ^\pm$ are formally given by \eqref{def-opD}.  Their relationship to the boundary condition 
is given in the next lemma.
\begin{lemma}\label{lem-on-D}
Let $g(x,v)=g(r,v_r,v_\theta)$ be a $C^1$ radial function on $\overline\DD\times\RR^2$.  Then $g$ satisfies the specular boundary condition 
if and only if  
$$ \int_\DD \int_{\RR^2} g\ \oD^\pm h\ dvdx  =  -  \int_\DD \int_{\RR^2} \oD^\pm g\ h\ dvdx $$
(either + or -) for all radial $C^1$ functions $h$ with $v$-compact support that satisfy the specular condition.   
\end{lemma}
\begin{proof} Integrating by parts in $x$ and $v$, we have 
$$ 
\int_\DD  \int_{\RR^2} \{g\ \oD ^\pm h  +  \oD^\pm g\ h\} \ dvdx =  2\pi\int g h\ \hat v \cdot e_r \; dv\Big|_{r=1}.  $$
If $g$ satisfies the specular condition, then $g$ and $h$ are even functions of $v_r=v\cdot e_r$ on $\partial\DD$, 
so that the last integral vanishes.  Conversely, if $\int g h\ \hat v \cdot e_r \; dv=0$ on $\partial\DD$, then 
$\int g(1, v_r, v_\theta)\ k(v_r,v_\theta) dv=0$ for all test functions that are odd in $v_r$, 
so it follows that $g(1,v_r,v_\theta)$ is an even function of $v_r$.  
\end{proof}

Therefore we {\it define} the domain of $\oD^\pm$ to be 
\begin{equation}\label{domainD} 
 \text{dom}(\oD^\pm)  = \Big \{g\in\cH ~\Big|~\ \oD^\pm g \in \cH,\ 
 \langle\oD^\pm g,h\rangle_\cH  =  -  \langle g,\oD^\pm h\rangle_\cH, \  
\forall h\in \mathcal C \Big\},  
\end{equation}
where $\mathcal C$ denotes the set of radial $C^1$ functions $h$ with $v$-compact support that satisfy the specular condition.  
We say that a function $g\in\cH$ with $ \oD^\pm g \in \cH$ satisfies the {\it specular boundary condition in the weak sense} 
if $g\in$ dom$(\oD^\pm)$.  Clearly, $ \text{dom}(\oD^\pm)$ is dense in $\cH$ since by Lemma \ref{lem-on-D} it contains the space $\mathcal{C}$ of test functions, which is of course dense in $\cH$. 

It follows that 
\begin{equation}\label{skewadj}
 \langle\oD^\pm g,h\rangle_\cH  =  -  \langle g,\oD^\pm h\rangle_\cH  \end{equation}
for all $g,h\in \text{dom}(\oD^\pm)$.  Indeed, given $h\in \text{dom}(\oD^\pm)$, we just approximate 
it in $\cH$ by a sequence of test functions in $\mathcal{C}$, and so \eqref{skewadj} holds thanks to Lemma \ref{lem-on-D}. 

Furthermore, with these domains, $\oD^\pm$ are {\it skew-adjoint} operators on $\cH$. 
Indeed, the skew-symmetry has just been stated.  To prove the skew-adjointness of $\oD^+$, 
suppose that $f,g\in\cH$ and $ \langle f,h\rangle_\cH  =  -  \langle g,\oD^+ h\rangle_\cH$  
for all  $h\in \text{dom}(\oD^\pm)$.  
Taking $h\in \mathcal{C}$ to be a test function, we see that $\oD^+ g = f $ in the sense of distributions.  
Therefore \eqref{skewadj} is valid for all such $h$, which means that $g\in\text{dom}(D^+)$.  

\subsection{Growing modes}\label{ss-growing}
Now we can state some necessary properties of any growing mode. 
Recall that by definition a growing mode satisfies  $f^\pm \in \cH $ 
and $\vE, \vB\in L^2(\DD)$. 
\begin{lemma} 
\label{growprops}  
Let $(e^{\lambda t} f^\pm, e^{\lambda t}{\vE}, e^{\lambda t} {\vB})$ with $\R\lambda>0$ 
be any growing mode.  Then 
$\vE,\vB\in H^1(\DD)$ and 
$$
\iint_{\DD\times \RR^2} \ (|f^\pm|^2  +  |\oD^\pm f^\pm|^2)\  \frac {dvdx}{|\mu_e^\pm|}  <  \infty.$$
\end{lemma}
\begin{proof}
The fields are given by  \eqref{def-EB} where $\varphi, \psi$ satisfy the elliptic system \eqref{Maxwell-eqs} 
with Dirichlet or Neumann boundary conditions, expressed weakly.  
The densities $f^\pm$ satisfy \eqref{lin-VM}.  Explicitly, 
\begin{equation}\label{Vlasov3} 
 \lambda f^\pm + \oD ^\pm f^\pm  
 =  \pm\mu_e^\pm \hat v_r\D_r\varphi  \pm  \mu_p^\pm \hat v_r\D_r(r\psi)  
 \pm  \lambda(\mu_e^\pm\hat v_\theta + r\mu_p^\pm) \psi. 
 \end{equation}
 This equation implies that $\oD^\pm f^\pm\in \cH$.  
 The specular boundary condition on $f^\pm$ is expressed 
 weakly by saying that $f^\pm\in\text{dom}(\oD^\pm)$.  
Dividing  by $|\mu_e^\pm|$ and defining  $g^\pm  = f^\pm/ |\mu_e^\pm|$,  we write 
the equation in the form $\lambda g^\pm + \oD^\pm g^\pm = h^\pm$, where the right side 
$h^\pm$   belongs to $\cH = L^2_{ |\mu_e^\pm|}(\DD\times \RR^2)$ thanks to the decay assumption \eqref{mu-cond} on $\mu^\pm$.  

Letting 
$w_\ep =  |\mu_e^\pm|/(\ep  +   |\mu_e^\pm|)$ for $\ep>0$ and $g_\ep = w_\ep g^\pm$, we have 
 $\langle \lambda g_\ep + \oD^\pm g_\ep , g_\ep \rangle_\cH 
= \langle w_\ep h^\pm, g_\ep\rangle_\cH. $  
It easily follows that 
$g_\ep\in\cH$.  In fact, $g_\ep\in\text{dom}(\oD^\pm)$, which means that the specular boundary condition holds 
in the weak sense, so that \eqref{skewadj} is valid for it.  In  \eqref{skewadj} we take both functions to be $g_\ep$ 
and therefore $\langle \oD^\pm g_\ep , g_\ep \rangle_\cH  = 0$.  It follows that 
$$| \lambda| \|g_\ep\|_\cH^2  = | \langle w_\ep h^\pm, g_\ep\rangle_\cH |\le \|h^\pm\|_{\cH} \|g_\ep\|_{\cH}. $$ Letting $\ep\to0$, we infer that $g^\pm \in \cH$, which means that 
$\iint |f^\pm|^2 / |\mu_e^\pm| dvdx < \infty$.    

Now the elliptic system for the field is  
$$-\Delta\varphi = \int (f^+-f^-)dv, \quad (\lambda^2 -\Delta_r)\psi = \int \hat v_\theta (f^+-f^-)dv $$ 
together with the boundary conditions $\D_r \varphi(t,1) = 0, \  \psi(t,1) =0$, which are expressed weakly.  
Because of $\iint |f^\pm|^2 / |\mu_e^\pm| dvdx < \infty$, 
the right sides of this system are now known to be finite a.e. and to belong to $L^2(\DD)$.  
So  it follows that $\psi, \varphi \in H^2(\DD)$ and  $\vE, \vB\in H^1(\DD)$.  This is the first assertion of the lemma.   
Nevertheless, we emphasize that $\oD^\pm f^\pm$ does not satisfy the specular boundary condition.  
However, directly from \eqref{Vlasov3} it is now clear that $\iint |\oD^\pm f^\pm|^2 / |\mu_e^\pm| dvdx < \infty  $. 
This is the last assertion.  
\end{proof}

\section{Linear stability}


\subsection{Formal argument}\label{sec-formal} 
Before presenting the stability proof, let us sketch a formal proof. 
We consider the linearized RVM system \eqref{lin-VM}. For sake of presentation, let us consider the case with one particle $f = f^+$, and thus drop all the superscripts $+$. We have the linearized equation: 
$$ 
\dt f +\oD f =~~\mu_e \oD \varphi  + \mu_p \oD  (r \psi) + (\mu_e \hat v_\theta+ r \mu_p) \D_t \psi.$$
The key ingredient for stability is the fact that the functional  
$$
\begin{aligned}\mathcal{I}(f,\varphi,\psi) &:= \int _\DD  \int \Big[ \frac {1}{|\mu_e|} |f - r \mu_p \psi|^2 - r \mu_p \hat v_\theta |\psi|^2 \Big] \; dvdx + \int_\DD \Big[ |\vE|^2 + |\vB|^2 \Big] \; dx  
\end{aligned}$$ 
is time-invariant, which can be found by formally expanding the usual nonlinear energy-Casimir functional around the equilibria. Next, we then write the linearized equation in the form 
\begin{equation}\label{conserved-formVM}\begin{aligned}
\dt (f -  \mu_e \hat v_\theta \psi- r\mu_p \psi) +\oD  (f - \mu_e \varphi - r\mu_p  \psi) =0.
\end{aligned}\end{equation}
We observe that $(f -  \mu_e \hat v_\theta \psi - r\mu_p \psi)$ stays orthogonal to $\ker \oD $ for all time, and that in case of stability we would expect that $f - \mu_e \varphi - r\mu_p  \psi$ asymptotically belongs to $\ker \oD $. That is, if $\mathcal{P}$ is the projection onto the kernel, one would have at large times 
$$\begin{aligned} f - \mu_e \varphi - r\mu_p  \psi = \mathcal{P}(f - \mu_e \varphi - r\mu_p  \psi), \qquad \mathcal{P} (f -  \mu_e \hat v_\theta \psi - r\mu_p \psi) =0.
 \end{aligned}
$$
Adding up these identities, we obtain $$f = \mu_e (1-\mathcal{P})\varphi + r\mu_p  \psi + \mu_e \mathcal{P}(\hat v_\theta\psi).$$ This asymptotic description of $f$ is 
essential to the proof of stability, which we will prove rigorously in subsection \ref{ss-min}; see \eqref{id-f0}. Next, by plugging the identity into the functional $\mathcal{I}(f,\varphi, \psi)$, 
we will obtain $$\begin{aligned}\mathcal{I}(f,\varphi, \psi) = (\mathcal{A}_1^0\varphi,\varphi )_{L^2} + ( \mathcal{A}_2^0 \psi, \psi)_{L^2}+  \int_\DD  |\dt \psi|^2\; dx,
\end{aligned}$$
in which the operators $\mathcal{A}_j^0$ are defined as in \eqref{operators-0}.
Using the identity $ \mathcal{A}_1^0\varphi = \mathcal{B}^0 \psi$, which is obtained from the Poisson equation, to eliminate $\varphi$, we can formally write
\begin{equation}\label{key-identity}
\begin{aligned}\mathcal{I}(f,\varphi, \psi) = ( \mathcal{L}^0 \psi, \psi)_{L^2}+  \int_\DD  |\dt  \psi|^2\; dx,
\end{aligned}\end{equation} 
with $\mathcal{L}^0 := \mathcal{A}_2^0 + (\mathcal{B}^0)^* (\mathcal{A}_1^0)^{-1} \mathcal{B}^0$. (In fact, we are only able to prove a reverse inequality $(\ge)$, which however suffices for stability). Now, if we assumed that there were a growing mode of the linearized system, then the time-invariant functional $\mathcal{I}(f,\varphi, \psi)$ must be zero. Thus \eqref{key-identity} shows that $\mathcal{L}^0 \not \ge 0$. That is, $\mathcal{L}^0\ge0$ formally implies the stability. 

\subsection{Key operators} \label{ss-keyops} 
 In this subsection, we shall derive the basic properties of 
the operators $\mathcal{A}_j^0$ and $\mathcal{B}^0$ defined in \eqref{operators-0}.  Let us recall that $\mathcal{P}^\pm$ are the orthogonal projections of $\cH$ onto the kernels 
$$\ker(\oD^\pm) = \Big\{ f \in \text{dom}(\oD^\pm) ~\Big|~ \oD^\pm f = 0\Big\},$$
and the key operators:
\begin{equation}\label{operators-0re}\begin{aligned}  \mathcal{A}_1^0 \varphi & = - \Delta \varphi  - \int \mu^+_e(1-\mathcal{P}^+) \varphi \; dv - \int \mu^-_e (1-\mathcal{P}^-) \varphi \; dv
\\
 \mathcal{A}_2^0 \psi & =- \Delta_r \psi   -  \int r\hat v_\theta (\mu_p^+ + \mu_p^-) \; dv \psi  -   \int \hat v_\theta\Big(\mu_e^+ \mathcal{P}^+(\hat v_\theta \psi) +\mu_e^- \mathcal{P}^-(\hat v_\theta \psi) \Big)\; dv
\\
\mathcal{B}^0 \psi  &= r \int \Big(\mu_p^+ +  \mu_p^-\Big )\; dv\psi + \int \Big(\mu_e^+\mathcal{P}^+(\hat v_\theta \psi) +\mu_e^-\mathcal{P}^-(\hat v_\theta \psi) \Big) \; dv .
\end{aligned}\end{equation}
We also recall that the functional space $\mathcal{V}$ consists of functions in $H^2_r(\DD )$ that satisfy the Neumann boundary condition and have zero average over $\DD$, whereas $\mathcal{V}^\dagger$ consists of functions in $H^{2\dagger}(\DD)$ that satisfy the Dirichlet condition on $\D\DD$. 

\begin{lemma}\label{lem-AB0property}~

(i) $\mathcal{A}_1^0$ is self-adjoint and positive definite on $L^2_r(\DD )$ with  domain $\mathcal{V}$. 
 $\mathcal{A}_1^0$ is a one-to-one map from $\mathcal{V}$ onto the set $\{g \in L^2_r\ |\  \int_\DD g \; dx =0\}$.
 In particular, $(\mathcal{A}_1^0)^{-1}$ is well-defined on the range of $\mathcal{B}^0$. 

(ii) $\mathcal{B}^0$ is a bounded operator on $L^2_r(\DD )$. 

(ii) $\mathcal{A}_2^0$ and $\mathcal{L}^0$ are self-adjoint on $L^2_r(\DD )$ with common domain $\mathcal{V}^\dagger$. 

\end{lemma}

\begin{proof} First, since the projections $\mathcal{P}^\pm$ preserve the radial symmetry, the operators $\mathcal{A}_j^0$ and $\mathcal{B}^0$ also preserve the symmetry. Next we observe that all the integral terms in \eqref{operators-0re} are bounded operators in $L^2_r(\DD )$. For example, we have 
$$
\Big| \int_\DD  \int \mu_e^+\psi \mathcal{P}^+\varphi \; dvdx\Big| \le  \sup_\DD  \Big(\int |\mu_e^+|\; dv \Big) \|\psi\|_{L^2} \|\mathcal{P}^+\varphi\|_{L^2}\le C_0  \|\psi\|_{L^2} \|\varphi\|_{L^2},$$ 
for some constant $C_0$ that depends on the decay assumption \eqref{mu-cond} on $\mu_e^+$.  
The other integrals are similar since $\hat v_\theta$ is bounded by one.  
This proves (ii) and also proves that the integral terms in $\mathcal{A}_1^0$ and $\mathcal{A}_2^0$ 
are relatively compact with respect to $-\Delta$ and $-\Delta_r$, 
which have domains $\mathcal V$ and $\mathcal V^\dagger$, respectively. 
Thus, $\mathcal{A}_1^0$ and $\mathcal{A}_2^0$ are well-defined operators on $L^2_r(\DD )$ 
with  domains $\mathcal{V}$ and $\mathcal{V}^\dagger$, respectively. \bigskip

Since $\mathcal{P}^\pm$ are self-adjoint on $\cH$, it is clear that all three operators 
$\mathcal{A}_1^0,\mathcal{A}_2^0,$ and $\mathcal{L}^0$ are self-adjoint  on $L^2_r(\DD )$. 
Note that the  function 1 belongs to the kernel of $\oD ^\pm$.
To prove the positivity of $\mathcal{A}_1^0$, we use the orthogonality of $\mathcal{P}^\pm$ and $1-\mathcal{P}^\pm$ 
in $\cH$ to write 
$$
 -\int_\DD  \int \mu^\pm_e\varphi (1-\mathcal{P}^\pm) \varphi \; dvdx 
 =  -\int_\DD  \int \mu^\pm_e |(1-\mathcal{P}^\pm) \varphi|^2 \; dvdx \ge0$$ 
since $\mu_e^\pm <0$.   
Thus $\mathcal{A}_1^0$ is a nonnegative operator,  and $\mathcal{A}_1^0\varphi =0$ if and only if $\varphi$ is a constant.  
Since $\mathcal{A}_1^0$ has discrete spectrum, it is invertible on the orthogonal set to its kernel,  
that is,  on $\{g \in \mathcal{V}\ | \  \int_\DD g \; dx =0\}$.  
In order to prove the invertibility of $\mathcal{A}_1^0$ on the range of $\mathcal{B}^0$, 
we note that by the self-adjointness 
of $\mathcal{P}^\pm$ in $\cH$, we have  
$$\int_\DD\int \mu_e^\pm\mathcal{P}^\pm(\hat v_\theta \psi) \; dvdx 
= \int_\DD\int \mu_e^\pm \hat v_\theta \psi \mathcal{P}^\pm(1) \; dvdx= \int_\DD\int \mu_e^\pm \hat v_\theta \psi \; dvdx.$$
Thus 
$$\int_\DD \mathcal{B}^0 \psi \; dx = \sum_\pm \int_\DD\int ( r \mu_p^\pm + \mu_e^\pm \hat v_\theta )\psi \; dvdx ,$$
which is identically zero by using the fact that $\partial_{v_\theta}[ \mu^\pm ] = \hat v_\theta \mu_e^\pm + r \mu_p^\pm$. 
That is, $\mathcal{B}^0\psi$ has zero average, and so $(\mathcal{A}^0_1)^{-1}$ is well-defined on the range of $\mathcal{B}^0$.   
\end{proof}

\subsection{Invariants} \label{inv} 
First we consider the linearized energy.  
\begin{lemma} \label{invariance} 
Suppose that  $(f^\pm, \varphi,\psi)$ is a solution of our linearized system  \eqref{lin-VM}, \eqref{BC-Maxwell}, and  
\eqref{bdry-specular} such that 
$f^\pm\in C^1(\RR, L^2_{1/|\mu_e|}(\DD\times\RR^2))$ and $\varphi, \psi \in C^1(\RR; H^2(\DD))$.  
Then the linearized energy functional 
$$\begin{aligned}&\mathcal{I}(f^\pm,\varphi, \psi ) \\
&:= \sum_\pm \int _\DD  \int \Big[ \frac {1}{|\mu^\pm_e|} |f^\pm \mp r \mu^\pm_p  \psi |^2 
- r \mu^\pm_p \hat v_\theta | \psi |^2 \Big] \; dvdx 
+  \int_\DD \Big[ |\dt \psi|^2 + |\nabla\varphi|^2 + \frac 1{r^2}| \D_r (r\psi)|^2 \Big] \; dx\end{aligned}$$
is independent of time.  
\end{lemma} 

\begin{proof} 
For convenience of calculation, we denote the three terms as 
$$\begin{aligned}\mathcal{I}^\pm_V(f^\pm, \psi ) 
&:=  \int _\DD  \int \Big[ \frac {1}{|\mu^\pm_e|} |f^\pm \mp r \mu^\pm_p  \psi |^2  
- r \mu^\pm_p \hat v_\theta | \psi |^2 \Big] \; dvdx \\
\mathcal{I}_M(\varphi, \psi )&:=   \int_\DD  \Big[ |\nabla \varphi|^2 + |\D_t\psi |^2 +  \frac 1{r^2}| \D_r (r\psi)|^2 \Big] \; dx, 
\end{aligned}$$
so that 
 $$
 \mathcal{I}(f^\pm,\varphi, \psi )= \mathcal{I}^+_V(f^+, \psi )+\mathcal{I}^-_V(f^-, \psi )+ \mathcal{I}_M(\varphi, \psi ).$$
Now  
taking the time derivative of $\mathcal{I}^+_V(f^+, \psi )$ and then using the linearized Vlasov equation 
\eqref{lin-VM} for $f^+$, we get 
$$\begin{aligned}
&\frac 12\frac {d}{dt}\mathcal{I}^+_V(f^+, \psi ) \\&  
= \int _\DD  \int \left[ \frac {1}{|\mu^+_e|} (f^+ - r \mu^+_p  \psi )(\dt f^+ - r \mu^+_p \dt  \psi )  
- r \mu^+_p \hat v_\theta  \psi  \dt  \psi  \right] \; dvdx 
\\
 &= \int _\DD  \int \left[ \frac {1}{|\mu^+_e|} (f^+ - r \mu^+_p  \psi )
 \Big\{-\oD ^+ f^+ +\mu^+_e\oD ^+ \varphi  + \mu^+_p\oD ^+(r \psi )  +  \mu^+_e \hat v_\theta \dt  \psi  \Big\} 
 - r \mu^+_p \hat v_\theta  \psi  \dt  \psi  \right] \; dvdx 
\\
 &= \int _\DD  \int \frac {1}{|\mu^+_e|}   \Big [ - f^+\oD ^+f^+ 
 + \mu^+_e f^+\oD ^+\varphi + \mu^+_p f^+\oD ^+(r \psi ) + \mu^+_e  f^+ \hat v_\theta  \dt  \psi  + r \mu^+_p  \psi \oD ^+ f^+ \\ 
 &\qquad\qquad - r \mu^+_p \mu^+_e  \psi \oD ^+\varphi -( \mu^+_p)^2 r \psi \oD ^+(r \psi )    \Big ] \; dvdx . 
\end{aligned}$$
Among the nine terms, we have used the fact that two terms with $\partial_t\psi$ exactly cancel because $\mu_e<0$.  
Some of the remaining seven terms cancel, as follows.  First, we observe that the sixth term  vanishes, 
upon writing  $ \oD ^+\varphi = \hat v_r \D_r \varphi $ and noting that $\mu$ is even in $v_r$. 
Next, by using the skew--symmetric property \eqref{skewadj} of $\oD ^+$, 
which is the weak form of the specular boundary condition, the first and seventh terms are also zero.  
Now the third and fifth terms can be combined to get 
  $\mu^+_p\oD ^+(r \psi  f^+) $, whose integral again vanishes due to the boundary condition $\psi=0$. 
We have used the fact  that both $\mu_e^+$ and $\mu_p^+$ belong to the kernel of $\oD ^+$.   
Only the second and fourth terms survive, so we  have 
$$\begin{aligned}
\frac 12&\frac {d}{dt}\mathcal{I}^+_V(f^+, \psi ) 
=\int _\DD  \int \Big[- f^+ \oD ^+\varphi - f^+ \hat v_\theta  \dt  \psi \Big] \; dvdx = \int _\DD  \int \vE \cdot \hat v\ f^+\; dvdx,
\end{aligned}$$
in which we have used the fact that $\oD ^+\varphi = \hat v_r \D_r \varphi$ together with the definition $\vE = -\D_r \varphi e_r- \D_t \psi e_\theta $. Entirely similar calculations hold for $f^-$. We therefore obtain 
\begin{equation}\label{est-IV}\begin{aligned}
\frac 12\frac {d}{dt}\Big(\mathcal{I}^+_V(f^+, \psi ) + \mathcal{I}^-_V(f^-, \psi )\Big)&=\int _\DD  \int \vE \cdot \hat v ( f^+ - f^-)\; dvdx = \int_\DD  \vE \cdot \vj \; dx.
\end{aligned}\end{equation}

Next, by \eqref{def-potentials} we may write 
$\mathcal{I}_M(\varphi, \psi )=   \int_\DD  \Big[ |\vE|^2 + |\vB|^2  \Big] \; dx.$
Taking its time derivative
and using the Maxwell equations, we obtain
$$\begin{aligned}
\frac12\frac {d}{dt}\mathcal{I}_M(\varphi, \psi )&= \int_\DD  \Big[ \vE \cdot \D_t \vE + \vB \cdot \D_t \vB \Big] \; dx
\\
&= \int_\DD  \Big[ - \vE \cdot \vj + \vE \cdot (\nabla \times \vB) - \vB \cdot (\nabla \times \vE) \Big] \; dx.
\\
&= -\int_\DD   \vE \cdot \vj \; dx + \int_{\D\DD} (\vE \times \vB) \cdot n(x) \; dS_x 
= -\int_\DD   \vE \cdot \vj \; dx 
\end{aligned}$$
Here we have  used the fact that $(\vE \times \vB) \cdot n = (\vE \times n) \cdot \vB $, which vanishes  
due to the perfect conductor boundary condition $\vE\times n =0$. 
Together with \eqref{est-IV}, this yields invariance of $\mathcal{I}(f^\pm,\varphi, \psi )$. 
\end{proof}
\bigskip

Furthermore, we also obtain the following. 
\begin{lemma} \label{lem-invK} 
Suppose that  $(f^\pm, \varphi,\psi)$ is a solution of our linearized system  \eqref{lin-VM}, \eqref{BC-Maxwell}, and  
\eqref{bdry-specular} such that 
$f^\pm\in C^1(\RR, L^2_{1/|\mu_e|}(\DD\times\RR^2))$ and $\varphi, \psi \in C^1(\RR; H^2(\DD))$.  
Then the functional 
\begin{equation}\label{inv-K}  
\mathcal{K}^\pm_g(f^\pm, \psi )
 = \int_\DD  \int \Big( f^\pm \mp \mu^\pm_e \hat v_\theta  \psi  \mp r \mu^\pm_p  \psi \Big) g \; dvdx\end{equation} 
 are independent in time, for all $g \in \ker\oD^\pm $ and for both + and $-$.  In particular, for $g=1$ in \eqref{inv-K}, the integrals 
$ \int_\DD  \int_{\RR^2} f^\pm (t,x,v)\; dvdx $ are time-invariant; that is, the total masses are conserved. 
\end{lemma} 
\begin{proof} Parenthetically, we remark that such invariants $\mathcal{K}^\pm_g(f^\pm, \psi )$ can easily be discovered by writing the Vlasov equations 
as in  \eqref{conserved-formVM}.   
 Indeed, writing the Vlasov equations in the form \eqref{conserved-formVM} and 
 using the skew-symmetry property of $\oD ^\pm$ as in 
 \eqref{skewadj}, we have 
$$\begin{aligned}
\frac{d}{dt}\mathcal{K}^\pm_g(f^\pm, \psi ) &= \int_\DD  \int \D_t\Big( f^\pm \mp \mu^\pm_e \hat v_\theta  \psi  \mp r \mu^\pm_p  \psi \Big) g \; dvdx   \\
&= \int_\DD  \int \oD ^\pm  \Big( -f^\pm  \pm \mu^\pm_e\varphi \pm r\mu^\pm_p \psi \Big) g \; dvdx   \\
&= -  \int_\DD\int  \Big( -f^\pm \pm \mu^\pm_e \varphi \pm r \mu^\pm_p  \psi \Big) \oD^\pm g\; dvdx = 0 
\end{aligned}
$$
due to the specular conditions on $f^\pm$ and $g$ and the evenness in $v_r$ of $\mu^\pm$. Now, if we take $g=1$ in \eqref{inv-K} and note that 
$\D_{v_\theta}\mu^\pm = \mu^\pm_e \hat v_\theta + r \mu^\pm_p$, the integrals $ \int_\DD  \int_{\RR^2} f^\pm (t,x,v)\; dvdx $ are therefore time-invariant. 
\end{proof}

\subsection{Minimization} \label{ss-min}
In this subsection we prove an identity that will be fundamental to the proof of stability.  
It involves the functional 
$$\begin{aligned}\mathcal{J}_ {\psi }(f^+,f^-) & 
:= \int _\DD  \int \frac {1}{|\mu^+_e|} |f^+ - r \mu^+_p  \psi |^2 \; dvdx + 
\int _\DD  \int \frac {1}{|\mu^-_e|} |f^- + r \mu^-_p  \psi |^2 \; dvdx +  \int_\DD  |\nabla \varphi|^2 \; dx,  
\end{aligned}$$
where $\varphi = \varphi(r)\in \mathcal{V}$ satisfies the Poisson equation 
\begin{equation}\label{f-to-phi}  
-\Delta \varphi = \int (f^+ - f^-)\; dv, \qquad \int_\DD \varphi\; dx =0, \qquad \D_r\varphi (1)=0.
\end{equation} 
For each fixed $ \psi \in L^2_r(\DD )$, let $\mathcal{F}_{\psi }$ be the  space consisting of all pairs of measurable functions $(f^+,f^-)$ depending on  $(r,v)$ 
which satisfy the  constraints 
\begin{equation}\label{eqs-constraint-f} 
\int_\DD  \int \frac{1}{|\mu^\pm_e|} |f^\pm|^2 \; dvdx <+\infty,\end{equation}
 and
\begin{equation}\label{eqs-constraint} 
\int_\DD  \int \Big( f^+ -  \mu^+_e \hat v_\theta \psi  - r \mu^+_p  \psi \Big)\ g^+ \; dvdx = 0,   
\qquad \int_\DD  \int \Big( f^- + \mu^-_e \hat v_\theta \psi  + r \mu^-_p  \psi \Big)\ g^- \; dvdx = 0 ,
\end{equation}
for all $g^\pm \in \ker\oD ^\pm.$ 
Similarly,  let $\mathcal{F}_0$ be the space of pairs $(f^+,f^-)$ satisfying \eqref{eqs-constraint-f} and 
\begin{equation}\label{eqs-constraint-0} 
\int_\DD  \int  f^+ g^+ \; dvdx = 0,\qquad \int_\DD  \int  f^- g^- \; dvdx = 0, \qquad \forall ~g^\pm \in \ker\oD ^\pm.
\end{equation}
Note that the constraints in \eqref{eqs-constraint} with $g=1$ imply that for such a pair of functions $f^\pm$, there is a unique solution $\varphi\in \mathcal{V}$ of the Poisson problem \eqref{f-to-phi}.  
In particular, $\varphi$ is radially symmetric since $f^\pm$ are radially symmetric. 
Thus the functional $\mathcal{J}_{\psi }$ is well-defined and nonnegative on $\mathcal{F}_{\psi }$, and 
 its infimum over $\mathcal{F}_{\psi }$ is finite.   
 We next show that it indeed admits a minimizer on $\mathcal{F}_{\psi }$. 

\begin{lemma}\label{lem-estJ} 
For each fixed $ \psi \in L^2_r(\DD )$, there exists a pair of functions $f^\pm_*$ that 
minimizes the functional $\mathcal{J}_{\psi }$ 
on $\mathcal{F}_{\psi }$.   
Furthermore, 
\begin{equation}\label{J-calculation} 
\begin{aligned}&\mathcal{J}_{\psi } (f^+_*,f^-_*)  
= ( (\mathcal{B}^0)^* (\mathcal{A}_1^0)^{-1} \mathcal{B}^0  \psi , \psi  )_{L^2} 
-  \int_\DD  \int \mu^+_e  |\mathcal{P}^+(\hat v_\theta\psi )|^2 \; dvdx 
 - \int_\DD  \int \mu^-_e  |\mathcal{P}^-(\hat v_\theta \psi )|^2 \; dvdx   . \end{aligned}   
\end{equation}

\end{lemma}
\begin{proof} 
Take a minimizing sequence $f^\pm_n$ in $\mathcal{F}_{\psi }$ 
such that $\mathcal{J}_{\psi }(f^+_n,f_n^-) $ converges to the infimum of 
$\mathcal{J}_{\psi }$.      Since $\{f^\pm_n\}$ are bounded sequences in $L^2_{1/|\mu_e^\pm|}$,  
the weighted $L^2$ space associated with the constraint \eqref{eqs-constraint-f}, 
there are subsequences with weak limits in $L^2_{1/|\mu_e^\pm|}$, which we denote by $f^\pm_*$. 
It is  clear that $f^\pm_*$ satisfy the constraints \eqref{eqs-constraint}, 
and so they belong to $\mathcal{F}_{\psi }$. That is, $(f^+_*,f_*^-)$ must be a minimizer. 

In order to derive the identity \eqref{J-calculation}, 
let the pair $(f_*,f^-_*)\in \mathcal{F}_{\psi }$ be the minimizer and 
let  $\varphi_*\in H^2_r(\DD)$ be the associated solution of  problem \eqref{f-to-phi} with $f^\pm=f^\pm_*$. 
For each $(f^+,f^-) \in \mathcal{F}_{\psi }$, we denote 
\begin{equation}\label{def-hf} h^+  := f^+ -  \mu^+_e \hat v_\theta \psi  - r\mu^+_p  \psi , \qquad h^-  := f^- +  \mu^-_e \hat v_\theta \psi + r\mu^-_p  \psi  .\end{equation}
In particular, $ h^\pm_* := f^\pm_* \mp  \mu^\pm_e \hat v_\theta  \psi  \mp r\mu^\pm_p  \psi $.  
It is clear that $(f^+,f^-) \in \mathcal{F}_{\psi }$ if and only if $(h^+,h^-) \in \mathcal{F}_0$. 
Since $\D_{v_\theta}[\mu^\pm] = \mu^\pm_e \hat v_\theta + r \mu^\pm_p$, we have
$$ \int (f^+ - f^-)\; dv = \int (h^+ - h^-)\; dv. $$
Thus if we let $\varphi$ be the solution of the Poisson equation \eqref{f-to-phi}, $\varphi$ is independent of the 
change of variables in \eqref{def-hf}.  Consequently, $(f^+_*,f^-_*)$ is a minimizer of $\mathcal{J}_{\psi }$ 
on $\mathcal{F}_{\psi }$ if and only if $(h^+_*,h^-_*)$ is a minimizer of the functional
$$\begin{aligned} 
\mathcal{J}_0(h^+,h^-) = \int _\DD  \int \frac {1}{|\mu^+_e|} | h^+ +\mu^+_e \hat v_\theta \psi |^2 \; dvdx 
+\int _\DD  \int \frac {1}{|\mu^-_e|} | h^- - \mu^-_e \hat v_\theta \psi |^2 \; dvdx +  \int_\DD  |\nabla \varphi|^2 \; dx, 
\end{aligned}$$
on $\mathcal{F}_0$.     By minimization, the first variation is  
\begin{equation}\label{eqs-Lag} \int _\DD  \int \frac {-1}{\mu^+_e} (h^+_* + \mu^+_e \hat v_\theta  \psi ) h^+ \; dvdx 
+ \int _\DD  \int \frac {-1}{\mu^-_e} (h^-_* - \mu^-_e \hat v_\theta  \psi ) h^- \; dvdx 
+ \int_\DD  \nabla \varphi_* \cdot \nabla \varphi \; dx =0,\end{equation}
for all $(h^+,h^-) \in \mathcal{F}_0$ where $\varphi$ solves \eqref{f-to-phi}. 
By the Neumann boundary condition on $\varphi$, we have
 $$
 \int_\DD  \nabla \varphi_* \cdot \nabla \varphi \; dx = - \int_\DD  \varphi_* \Delta \varphi \; dx = \int_\DD  \int \varphi_* (h^+-h^-) \; dvdx.$$
Adding this to the identity \eqref{eqs-Lag}, we obtain
\begin{equation}\label{zero-id} 
\int _\DD  \int \frac {-1}{\mu^+_e} (h^+_* + \mu^+_e \hat v_\theta \psi  - \mu^+_e \varphi_*)\  h^+ \; dvdx 
+  \int _\DD  \int \frac {-1}{\mu^-_e} (h^-_* - \mu^-_e \hat v_\theta \psi + \mu^-_e \varphi_*)\  h^- \; dvdx  =0  
\end{equation}
for all $(h^+,h^-) \in \mathcal{F}_0$. 
In particular, we can take $h^- =0$ in \eqref{zero-id} to get 
$$ 
\int _\DD  \int \frac {-1}{\mu^+_e} (h^+_* + \mu^+_e \hat v_\theta \psi  - \mu^+_e \varphi_*) \ h^+ \; dvdx=0$$
for all $h^+\in L^2_{1/|\mu_e^+|}$ satisfying $\int_\DD  \int  h^+ g^+ \; dvdx = 0,$ for all $g^+ \in \ker\oD ^+.$

We claim that this identity implies $h^+_* + \mu^+_e \hat v_\theta\psi  - \mu^+_e \varphi_* \in \ker\oD ^+$. 
Indeed, let 
$$ k_* = |\mu_e^+|^{-1} (h^+_* + \mu^+_e \hat v_\theta \psi  - \mu^+_e \varphi_*), \quad \ell = |\mu_e^+|^{-1} h^+ .$$
Using the inner product in $\cH = L^2_{|\mu_e^+|}$, we  have 
$$ \langle k_*, \ell \rangle_\cH = 0 \quad \forall \ell\in (\ker \oD^+)^\perp .  $$ 
Because $\oD ^+$ (with the specular condition)  is a skew-adjoint operator on $\cH$, we have 
$k_* \in (\ker \oD^+)^{\perp\perp} = \ker \oD^+$. 
Thus 
$$ \oD^+ \{ f_*^+ - r\mu_p^+\psi - \mu_e^+\varphi_* \} 
=  \oD^+ \{ h_*^+ + \mu_e^+\hat v_\theta\psi - \mu_e^+\varphi_* \} = \mu_e^+ \oD^+ k_*  =  0.  $$
This proves the claim.  
Similarly $\oD^- \{ f_*^- + r\mu_p^-\psi + \mu_e^-\varphi_* \} = 0$.  
Equivalently,  
$$\begin{aligned}
 f^+_* - r \mu^+_p  \psi - \mu^+_e \varphi_*  &= \mathcal{P}^+(f^+_* - r \mu^+_p  \psi - \mu^+_e \varphi_* ),
 \\
 f^-_* + r \mu^-_p  \psi + \mu^-_e \varphi_*  &= \mathcal{P}^-(f^-_* + r \mu^-_p  \psi + \mu^-_e \varphi_* ).
 \end{aligned}$$
On the other hand,  the constraints \eqref{eqs-constraint} can be written as 
$ \mathcal{P}^+(f^+_* - \mu^+_e \hat v_\theta\psi  - r \mu^+_p  \psi ) =0$ 
and $ \mathcal{P}^-(f^-_* + \mu^-_e \hat v_\theta\psi  + r \mu^-_p  \psi ) =0$.   
Combining these identities, we have  
\begin{equation}\label{id-f0}\begin{aligned}
f^+_* -  r\mu^+_p  \psi &=~~\mu^+_e (1-\mathcal{P}^+)\varphi_*  + \mu^+_e \mathcal{P}^+(\hat v_\theta \psi ),
\\
f^-_* +  r\mu^-_p  \psi &= -\mu^-_e (1-\mathcal{P}^-)\varphi_* - \mu^-_e \mathcal{P}^-(\hat v_\theta \psi ).
\end{aligned}\end{equation}
Thus, using the orthogonality of $\mathcal{P}^\pm$ and $(1-\mathcal{P}^\pm)$ in $\cH$, we compute 
 $$\begin{aligned} 
 \int _\DD  \int \frac {-1}{\mu^\pm_e} |f^+_* \mp r \mu^\pm_p  \psi |^2 \; dvdx 
&=  -\int _\DD  \int \mu^\pm_e|(1-\mathcal{P}^\pm)\varphi_*|^2\; dvdx 
- \int_\DD  \int \mu^\pm_e  |\mathcal{P}^\pm(\hat v_\theta \psi )|^2 \; dvdx 
\\
&=  -\int _\DD  \int \mu^\pm_e \varphi_* (1-\mathcal{P}^\pm)\varphi_*\; dvdx 
- \int_\DD  \int \mu^\pm_e  |\mathcal{P}^\pm(\hat v_\theta \psi )|^2 \; dvdx  .
 \end{aligned}$$
Inserting these identities 
into the definition of $\mathcal{J}_{\psi }(f^+_*,f^-_*)$, together with the fact from the boundary conditions that 
$ \int_\DD  |\nabla \varphi_*|^2 \; dx = -  \int_\DD  \varphi_* \Delta \varphi_* \; dx$, we obtain
$$\begin{aligned}\mathcal{J}_{\psi }(f^+_*,f^-_*) &=   \int _\DD  \varphi_* \Big[-\Delta \varphi_*   
-  \int \mu^+_e (1-\mathcal{P}^+)\varphi_*\; dv  -\int \mu^-_e (1-\mathcal{P}^-)\varphi_*\; dv  \Big] \; dx \\
&\qquad  -  \int_\DD  \int \mu^+_e  |\mathcal{P}^+(\hat v_\theta\psi )|^2 \; dvdx  
- \int_\DD  \int \mu^-_e  |\mathcal{P}^-(\hat v_\theta\psi )|^2 \; dvdx  .
 \end{aligned}$$
By the definition \eqref{operators-0re} of $\mathcal{A}_1^0$, 
the first group of integrals  simply equals $(\mathcal{A}_1^0\varphi_* , \varphi_* )_{L^2} $.  

Thus it remains to prove that $\mathcal{A}_1^0\varphi_* = \mathcal{B}^0 \psi $, 
because $\mathcal{A}_1^0$ is invertible on the range of $\mathcal{B}^0$  so that 
$\varphi_* = (\mathcal{A}_1^0)^{-1}\mathcal{B}^0 \psi $. 
Indeed, we plug the identities  \eqref{id-f0} into the Poisson equation \eqref{f-to-phi} for $\varphi_*$, 
resulting in the equation 
\begin{equation}\label{id-Poisson}\begin{aligned}
-\Delta \varphi_* & = \int \mu^+_e ( 1-\mathcal{P}^+)\varphi_* \; dv + \int \mu^-_e ( 1-\mathcal{P}^-)\varphi_* \; dv 
\\ &\quad+  r \int (\mu^+_p+\mu^-_p)\; dv   \psi  +\int  \Big(\mu^+_e \mathcal{P}^+(\hat v_\theta \psi ) 
+ \mu^-_e \mathcal{P}^-(\hat v_\theta \psi )\Big)\; dv.   
\end{aligned}\end{equation}
In view of the definitions of $\mathcal{A}_1^0$ and $\mathcal{B}^0$, this identity \eqref{id-Poisson} 
is equivalent to $\mathcal{A}_1^0\varphi_* = \mathcal{B}^0 \psi $, as desired.  
\end{proof}

\subsection{Growing modes are pure} \label{ss-pure}
In this subsection, we show that if 
$(e^{\lambda t} f^\pm, e^{\lambda t}\vE, e^{\lambda t} B)$ with $\R \lambda>0$ is a 
complex growing mode, then $\lambda$ must be real.  
See subsection \ref{ss-growing} for the properties of a growing mode.  
We now follow the splitting method in \cite{LS1} to show that $\lambda$ is real.  
Let $f_{\mathrm{ev}}$ and $f_{\mathrm{od}}$ be the even and odd parts  of $f$ 
with respect to the variable $v_r$.   
That is, we have the splitting: $ f = f_{\mathrm{ev}} + f_{\mathrm{od}}$, and furthermore, 
by inspection from the definition \eqref{def-opD}, the operators $\oD^\pm$ map even to odd functions 
and vice versa. We therefore obtain, from the Vlasov equations \eqref{lin-VM}, the split equations 
$$\left\{\begin{aligned}
\lambda f^+_{\mathrm{ev}} + \oD^+ f^+_{\mathrm{od}}  \quad 
&=\quad \lambda ( \mu^+_e \hat v_\theta  + r\mu^+_p)  \psi  = \lambda  \D_{v_\theta}[\mu^+] \psi,
\\
\lambda f^+_{\mathrm{od}} + \oD^+ f^+_{\mathrm{ev}}  \quad 
&=\quad \mu^+_e\oD ^+ \varphi + \mu^+_p\oD ^+(r\psi) .
\end{aligned}\right.  $$
By Lemma \ref{growprops}, we know that $ \iint |f^+|^2 /|\mu_e| dvdx<\infty$.     It follows that the same is true for 
$ f^+_{\mathrm{ev}}$, so from the first split equation we also have 
$ \iint |D^+f_{\mathrm{od}}^+|^2 /|\mu_e| dvdx<\infty$.  
The split equations imply that 
\begin{equation}\label{wave-fod} 
(\lambda^2  - {\oD^+}^2) f^+_{\mathrm{od}} 
= \lambda \mu_e^+ \oD^+ \varphi - \lambda \mu_e^+ \oD^+(\hat v_\theta \psi).                \end{equation}
Let $\overline  f^+$ be the complex conjugate of $f^+$.   
By the specular boundary condition on $f^+$ in its weak form \eqref{skewadj}, 
it follows that $f_{\mathrm{od}}^+$  satisfies the specular condition. 
 (Formally, $f_{\mathrm{od}}^+$ vanishes on the boundary  $\D\DD$.)
However, since $D^+ f_{\mathrm{od}}^+$ is even in the variable $v_r$, 
$D^+ f_{\mathrm{od}}^+$  also satisfies the specular condition.  
Thus when we multiply  equation \eqref{wave-fod} by $\overline f^+_{\mathrm{od}}  /  {|\mu_e^+|}$   
and integrate the result over $\DD \times \RR^2$, 
we may apply the skew-symmetry property \eqref{skewadj} of $\oD^+$.  
We obtain 
\begin{equation*} 
\lambda^2 \int_\DD \int \frac{1}{|\mu_e^+|} |f^+_{\mathrm{od}}|^2 \; dvdx 
+  \int_\DD \int \frac{1}{|\mu_e^+|} |\oD^+ f^+_{\mathrm{od}}|^2 \; dvdx  
=  - \int_\DD \int \Big( \lambda \oD^+\varphi   \overline f^+_{\mathrm{od}} 
+\lambda \hat v_\theta \psi  \oD^+ \overline f^+_{\mathrm{od}} \Big) \; dvdx.                \end{equation*}
 Similarly for $f^-$ we  obtain
\begin{equation*}\lambda^2 \int_\DD \int \frac{1}{|\mu_e^-|} |f^-_{\mathrm{od}}|^2 \; dvdx +  \int_\DD \int \frac{1}{|\mu_e^-|} |\oD^- f^-_{\mathrm{od}}|^2 \; dvdx  
=   \int_\DD \int \Big(  \lambda \oD^-\varphi   \overline f^-_{\mathrm{od}} 
+ \lambda \hat v_\theta \psi  \oD^- \overline f^-_{\mathrm{od}} \Big) \; dvdx. \end{equation*}
Adding  up 
these identities and examining the imaginary part of the resulting identity, we get 
\begin{equation}\label{key-id-fod}\begin{aligned}
2 \R &\lambda \I\lambda  \int_\DD \int \Big( \frac{1}{|\mu_e^+|} |f^+_{\mathrm{od}}|^2 
+ \frac{1}{|\mu_e^-|} |f^-_{\mathrm{od}}|^2\Big) \; dvdx \\
& =  - \I \int_\DD \int  \lambda (\oD^+ \varphi \overline f^+_{\mathrm{od}} 
- \oD^- \varphi \overline  f^-_{\mathrm{od}}) \; dvdx  
- \I \int_\DD\int \lambda \hat v_\theta \psi (\oD^+ \overline f^+_{\mathrm{od}}  
- \oD^- \overline  f^-_{\mathrm{od}} )\; dvdx 
.\end{aligned}                                                 \end{equation}

Let us now use the Maxwell equations to compute the terms on the right side of \eqref{key-id-fod}.   
First, we recall that the second equation in \eqref{Maxwell-eqs} is
$$\lambda \D_r \varphi = \int \hat v_r (f^+ - f^-)\; dv =  \int \hat v_r ( f^+_{\mathrm{od}} -  f^-_{\mathrm{od}}) \; dv.$$
Together with the definition of $\oD^\pm$, this yields
$$\begin{aligned} 
- \lambda \int_\DD \int  (\oD^+ \varphi \overline  f^+_{\mathrm{od}} - \oD^- \varphi \overline  f^-_{\mathrm{od}}) \; dvdx    
= - \lambda \int_\DD \D_r \varphi \Big( \int \hat v_r (\overline  f^+_{\mathrm{od}} - \overline  f^-_{\mathrm{od}}) \; dv \Big) dx  
= - |\lambda|^2 \int_\DD |\D_r \varphi |^2 dx,
 \end{aligned}$$
 whose imaginary part is identically zero.   
 
 Secondly,  using the Vlasov equation for $f^+_{\mathrm{ev}}$, we estimate
$$\begin{aligned}
- \lambda \int_\DD\int  \hat v_\theta \psi &(\oD^+ \overline  f^+_{\mathrm{od}} 
 - \oD^- \overline  f^-_{\mathrm{od}} )\; dvdx 
\\
&=-  \lambda \int_\DD\int  \hat v_\theta \psi(- \overline  \lambda \overline  f^+_{\mathrm{ev}} 
+ \overline  \lambda  \D_{v_\theta}[\mu^+] \overline  \psi 
+ \overline  \lambda \overline  f^-_{\mathrm{ev}} + \overline  \lambda \D_{v_\theta}[\mu^-] \overline  \psi )\; dvdx 
\\
&=   |\lambda|^2 \int_\DD\psi \Big(\int  \hat v_\theta (\overline  f^+_{\mathrm{ev}}   
- \overline  f^-_{\mathrm{ev}}\; dv \Big)dx   
-  |\lambda|^2 \int_\DD\int  |\psi|^2 \hat v_\theta ( \D_{v_\theta}[\mu^+] +  \D_{v_\theta}[\mu^-]) \; dvdx .  
 \end{aligned}$$
By \eqref{Maxwell-eqs} 
the first term on the right  equals  
$$  |\lambda|^2 \int_\DD\psi (- \Delta_r  + \overline  \lambda^2)\overline  \psi dx  =  |\lambda|^2 \int_\DD \frac 1{r^2}|\D_r (r\psi)|^2 \; dx + \overline  \lambda^2 |\lambda|^2 \int_\DD |\psi|^2 \; dx.$$
Here we integrated by parts and used the Dirichlet boundary condition on $\psi$.

Putting these estimates back into \eqref{key-id-fod}, we obtain 
$$2 \R \lambda \I\lambda  \int_\DD \int \Big( \frac{1}{|\mu_e^+|} |f^+_{\mathrm{od}}|^2 + \frac{1}{|\mu_e^-|} |f^-_{\mathrm{od}}|^2\Big) \; dvdx = - 2 \R \lambda \I \lambda |\lambda|^2 \int_\DD |\psi|^2 \; dx.
$$
The opposite signs of the integrals imply that $\lambda$ must be real. 

\subsection{Proof of stability} \label{ss-stab}
With the above preparations, we are ready to prove the following stability result, which is half of Theorem \ref{theo-main}.   
 \begin{lemma}\label{lem-stab} 
 If $\mathcal{L}^0\ge 0$, then there exists no growing mode $(e^{\lambda t} f^\pm, e^{\lambda t}{\vE}, e^{\lambda t} {B})$, 
 with $\R\lambda >0$. 
  \end{lemma}
\begin{proof} 
Assume that there were a growing mode $(e^{\lambda t}  f^\pm, e^{\lambda t}\vE, e^{\lambda t}  B )$. 
For the basic properties of any growing mode, see Lemma \ref{growprops}. 
By the result in the previous subsection, it is a purely growing mode, and thus we can assume that 
$(  f^\pm, \vE,  B)$ are real-valued functions.   By the time-invariance in Lemma \ref{invariance},  
the functional $\mathcal{I}(f^\pm,\varphi, \psi)$ must be identically equal to zero,  
where $\varphi$ and $\psi$ are defined as usual through the relations \eqref{def-potentials}. 
That is,  
$$\begin{aligned}{\mathcal{I}}( f^\pm,\varphi,  \psi) &
=  \mathcal{J}_{\psi }( f^+, f^-) - \int _\DD  \int r \hat v_\theta (\mu^+_p+\mu^-_p)  |  \psi|^2\; dvdx 
+ \int_\DD  \Big[ |\lambda|^2| \psi|^2  + \frac 1{r^2}| \D_r (r \psi)|^2\Big] \; dx  =  0.
\end{aligned}$$ 
Furthermore, all the expressions $\mathcal{K}^\pm_g( f^\pm, \psi)$ defined in \eqref{inv-K} must  be zero. 
The vanishing of the latter  integrals  is equivalent to the constraints 
in \eqref{eqs-constraint} and therefore  the pair $( f^+, f^-)$ belongs to the function space $\mathcal{F}_{\psi}$. 
We then apply the Lemma \ref{lem-estJ} to assert that   
$\mathcal{J}_\psi (f^+,f^-) \ge \mathcal{J}_\psi (f_*^+,f_*^-)$.  
Thus  we have 
\begin{equation}\label{lower-boundI}\begin{aligned}
{\mathcal{I}}( f^\pm, \varphi,  \psi) &  
\ge  ( (\mathcal{B}^0)^* (\mathcal{A}_1^0)^{-1} \mathcal{B}^0   \psi,  \psi )_{L^2} 
-  \int_\DD  \int \mu^+_e  |\mathcal{P}^+(\hat v_\theta   \psi)|^2 \; dvdx  
- \int_\DD  \int \mu^-_e  |\mathcal{P}^-(\hat v_\theta  \psi)|^2 \; dvdx  
\\&\qquad - \int _\DD  \int r \hat v_\theta (\mu^+_p+\mu^-_p) |  \psi|^2\; dvdx 
+ \int_\DD  \Big[ |\lambda|^2| \psi|^2  +\frac 1{r^2}| \D_r (r \psi)|^2 \Big] \; dx. 
\end{aligned}                                     \end{equation}
In addition, from \eqref{operators-0re} an integration by parts together with the Dirichlet boundary condition 
on $  \psi$ yields
$$\begin{aligned} 
( \mathcal{A}_2^0  \psi,  \psi)_{L^2} & = \int_\DD  \Big( |\D_r  \psi|^2 
+ \frac 1 {r^2} |  \psi|^2 \Big)\; dx - \int _\DD  \int r \hat v_\theta (\mu^+_p+\mu^-_p) |  \psi|^2\; dvdx 
\\&\qquad-  \int_\DD  \int \mu^+_e  |\mathcal{P}^+(\hat v_\theta  \psi)|^2 \; dvdx  
- \int_\DD  \int \mu^-_e  |\mathcal{P}^-( v_\theta  \psi)|^2 \; dvdx  
.\end{aligned}$$ 
Putting this calculation into \eqref{lower-boundI} and using the definition of $\mathcal{L}^0$, we then get 
$$\begin{aligned}
0 = {\mathcal{I}}( f^\pm, \varphi, \psi) &\ge  ( \mathcal{L}^0   \psi,  \psi )_{L^2} +  \int_\DD  |\lambda|^2| \psi|^2  \; dx.  
\end{aligned}$$
This is obviously a contradiction since we are assuming $\mathcal{L}^0\ge 0$. Thus there exists no growing mode for the linearized system.  
\end{proof}

\section{Linear instability} 
We now turn to the instability part of Theorem \ref{theo-main}. 
It is based on a spectral analysis of the relevant operators.   We plug the simple form 
$(e^{\lambda t} f^\pm, e^{\lambda t}\vE, e^{\lambda t} B)$, with $\lambda>0$, 
into the linearized RVM system \eqref{lin-VM} to obtain the Vlasov equations
\begin{equation}\label{Lap-VMsys}\left\{\begin{aligned}
(\lambda +\oD ^+ )f^+ \quad &=\qquad\mu^+_e\oD ^+ \varphi + \mu^+_p\oD ^+(r\psi) + \lambda ( \mu^+_e \hat v_\theta  + r\mu^+_p)  \psi 
\\
(\lambda  +\oD ^-) f^- \quad &=\quad -\mu^-_e\oD ^- \varphi  - \mu^-_p\oD ^-(r\psi)- \lambda ( \mu^-_e \hat v_\theta  + r\mu^-_p)  \psi 
\end{aligned}\right.
 \end{equation}
 and the Maxwell equations
 \begin{equation}\label{Lap-Maxwell}
 \left\{ \begin{aligned}
 - \Delta \varphi & =\int (f^+ - f^-) \;dv,
 \\
 ( - \Delta_r+\lambda^2) \psi &=   \int \hat v_\theta (f^+ - f^-)\; dv.
\end{aligned}\right.
\end{equation}
As before, we impose the specular boundary condition on $f^\pm$, the Neumann boundary condition on $\varphi$, 
and the Dirichlet boundary condition on $\psi $.  Recall that $(f^\pm,\varphi,\psi)$ is a perturbation of the equilibrium.  

\subsection{Particle trajectories}\label{ss-traj}
We begin  with the $+$ case (ions).  
 For each $(x,v) \in \DD \times \RR^2$, we introduce the particle trajectories $(X^+(s;x,v), V^+(s;x,v))$ defined by
\begin{equation}\label{trajectory} 
\dot X^+ = \hat V^+, \qquad
\dot V^+ = \vE^0(X^+) + \hat V^+\times \vB^0(X^+),
 \end{equation}
with initial values $ (X^+(0;x,v), V^+(0;x,v)) = (x,v).$ 
Because of the $C^1$ regularity of $\vE^0$ and $\vB^0$ in $\overline \DD$, each trajectory can be continued for 
at least a certain fixed time.  
Thus each particle trajectory exists and preserves cylindrical symmetry up to the first point where it meets the boundary.   
The trajectories reflect at the boundary $\D\DD$ at most a countable number of times as $s\to \pm\infty$.  
Let $s_0$ be any point at which the trajectory $X^+(s_0-;x,v)$ belongs to  $\D\DD $. 
In general, we write $h(s\pm)$ to mean the limit from the right (left).
By the specular boundary condition, 
the trajectory $(X^+(s;x,v), V^+(s;x,v))$ can be continued by the rule 
\begin{equation}\label{traj-reflection1}
(X^+(s_0+;x,v), V^+(s_0+;x,v)) = (X^+(s_0-;x,v), \tilde V^+(s_0-;x,v)),
\end{equation}  
with $\tilde V = (-V_r,V_\theta)$.  Thus $X^+$ is continuous and $V^+$ has a jump at $s_0$.  
Whenever the trajectory meets the boundary, it is reflected in the same way and then continued via the ODE \eqref{trajectory}. 
Such a continuation is guaranteed for some short time $s_1$ past $s_0$ by the standard ODE theory.   
Since $\vE^0$ and $\vB^0$ are $C^1$ smooth in $\overline \DD$, the additional time $|s_1-s_0|$ is bounded 
below by some fixed 
positive time $T_0$ independent of $s_0$. This shows that the trajectories can bounce at the boundary at most a 
countable number of times as $|s|\to \infty$. When there is no possible confusion, 
we will simply write $(X^+(s),V^+(s))$ or $(R^+(s), V_r^+(s), V_\theta^+(s))$ for the particle trajectories.  
The trajectories $(X^-(s),V^-(s))$ for the $-$ case (electrons) are defined similarly. 

\begin{lemma}\label{lem-trajectory} For each $(x,v)\in \DD \times \RR^2$, the particle trajectories 
$(X^\pm(s;x,v),V^\pm(s;x,v))$ are piecewise $C^1$ smooth in $s\in \RR$, and for each $s\in \RR$, 
the map $(x,v) \mapsto (X^\pm(s;x,v),V^\pm(s;x,v))$ is one-to-one and differentiable with  
Jacobian equal to one at all points $(x,v)$ such that $x\not \in \D\DD$ and $X^\pm(s;x,v) \not \in \D\DD$.   
In addition, the standard change-of-variables formula 
\begin{equation}\label{change-variable}
\int_\DD\int_{\RR^2} g(x,v) \; dxdv = \int_\DD\int_{\RR^2} g(X^\pm(-s;y,w),V^\pm(-s;y,w))\; dwdy
\end{equation} 
is valid for each $s\in \RR$ and for measurable functions $g$ for which the integrals are finite. 
\end{lemma}
\begin{proof} For each $(x,v)$, the particle trajectory $(X^\pm(s;x,v),V^\pm(s;x,v))$ is smooth except when it 
hits the boundary $\D\DD$, which happens  countably many times.  So the first assertion is clear.   
Given $s$,  let $\mathcal{S}$ be the set $(x,v)$ in $\DD \times \RR^2$ such that $X^\pm(s;x,v)\not \in \D\DD$. Clearly, $\mathcal{S}$ is open and its complement in $\DD \times \RR^2$ has Lebesgue measure zero. For each $s$, the trajectory map is one-to-one on $\mathcal{S}$ since the ODE \eqref{trajectory} and \eqref{traj-reflection1} are time-reversible and well-defined.   
In addition, a direct calculation shows that the Jacobian determinant is time-independent and is therefore equal to one.   The change-of-variable formula \eqref{change-variable} holds on the open set $\mathcal{S}$ and so on $\DD\times \RR^2$, as claimed. 
\end{proof}

\begin{lemma}\label{lem-reflection} Let $g(x,v)$ be a $C^1$ radial function on $\overline\DD\times\RR^2$.  
If $g$ is specular on $\D\DD$, then for all $s$, $g(X^\pm(s;x,v),V^\pm(s;x,v))$ is continuous and 
also specular on $\D\DD$.  That is, 
$$
g(X^\pm(s;x,v),V^\pm(s;x,v)) = g(X^\pm(s;x,\tilde v),V^\pm(s;x,\tilde v)), \qquad \forall ~(x,v)\in \D\DD\times \RR^2,$$
where $\tilde v = (-v_r,v_\theta)$ for all $v = (v_r,v_\theta)$. 
\end{lemma}
\begin{proof}
It follows directly by definition \eqref{trajectory} and \eqref{traj-reflection1} that for all $x\in \partial \DD$, 
the trajectory is unaffected by whether we start with $v$ or $\tilde v$.   So for all $s$ we have 
\begin{equation}\label{traj-reflection2}  
X^\pm(s;x,v) = X^\pm(s;x,\tilde v), \qquad V^\pm_\theta(s;x,v) 
= V^\pm_\theta(s;x,\tilde v),\qquad |V^\pm_r(s;x,v)| = |V^\pm_r(s;x,\tilde v)|.
 \end{equation}
In fact we have $V^\pm_r(s;x,v) = V^\pm_r(s;x,\tilde v)$ for any $s$ at which $X^\pm(s;x,v)\not \in \D\DD$, 
while $V^\pm_r(s+;x,v) = - V^\pm_r(s-;x,\tilde v)$ for $s$ at which $X^\pm(s;x,v)\in \D\DD$.  
Because $g$ is specular on the boundary, it takes the same value at $v_r$ and $-v_r$.  
Therefore $g(X^\pm(s),V^\pm(s))$ is a continuous function of $s$ at the points of reflection.   
It is specular because of  the rule \eqref{traj-reflection1}.
\end{proof}

\subsection{Representation of the particle densities}\label{ss-rep}
 We can now invert the operator $(\lambda +\oD ^\pm)$ to obtain an integral representation of $f^\pm$. 
By definition of the operator $\oD ^+$ from \eqref{def-opD} and the trajectories $(X^+(s),V^+(s))$ from \eqref{trajectory} and \eqref{traj-reflection1}, we have
$$\begin{aligned}
\int_{-\infty}^0 e^{\lambda s}\oD ^+ g (X^+(s),V^+(s)) \; ds 
&=  \int_{-\infty}^0 e^{\lambda s} \frac{d}{ds} g (X^+(s),V^+(s)) \; ds 
\\&=  g(x)  - \int_{-\infty}^0 \lambda e^{\lambda s}g (X^+(s),V^+(s)) \; ds ,\end{aligned}$$
for $C^1$ functions $g = g(x,v)$ which belong to $\text{dom}(\oD^+)$.   
Multiplying the Vlasov equations \eqref{Lap-VMsys} by $e^{\lambda s}$ and then integrating  
along the particle trajectories from $s=-\infty$ to zero, we readily obtain  
$$\begin{aligned} 
f^+(x,v)  &=  \mu^+_e \varphi  + r \mu^+_p\psi   - \mu^+_e \int_{-\infty}^0 \lambda e^{\lambda s} \varphi (R^+(s)) \; ds 
+ \int_{-\infty}^0 \lambda e^{\lambda s}\mu^+_e \hat V^+_\theta(s) \psi(R^+(s))  \; ds.
\end{aligned}$$
A similar derivation holds for the $-$ case. For convenience we denote 
$$ 
\mathcal{Q}^\pm_\lambda (g)(x,v)     
: = \int_{-\infty}^0 \lambda e^{\lambda s} g(X^\pm(s;x,v),V^\pm(s;x,v)) \; ds.    $$
In particular, by Lemma \ref{lem-reflection}, $\mathcal{Q}^\pm_\lambda(g)$ is specular on $\D\DD$ if $g$ is.  
Thus we have derived the integral representation for the particle densities:  
\begin{equation}\label{def-f}   
f^\pm(x,v)  =  \mu^\pm_e (1-\mathcal{Q}^\pm_\lambda)\varphi  
+ r \mu^\pm_p\psi   + \mu^\pm_e \mathcal{Q}^\pm_\lambda(\hat v_\theta \psi)
\end{equation}

\subsection{Operators} \label{ss-operators}
We next substitute \eqref{def-f} into the Maxwell equations \eqref{Lap-Maxwell}.   
We introduce the operators  
\begin{equation}\label{def-operators}\begin{aligned}  
\mathcal{A}_1^\lambda \varphi :& = - \Delta \varphi  -  \int \mu^+_e (1 - \mathcal{Q}^+_\lambda)\varphi \; dv 
-   \int \mu^-_e (1 - \mathcal{Q}^-_\lambda)\varphi \; dv  ,  
\\ \mathcal{A}_2^\lambda \psi :& =  (- \Delta_r + \lambda^2)  \psi  -  \int r\hat v_\theta(\mu^+_p + \mu^-_p) dv \psi  
-  \int \hat v_\theta \Big( \mu^+_e \mathcal{Q}^+_\lambda(\hat v_\theta \psi) 
+ \mu^-_e \mathcal{Q}^-_\lambda(\hat v_\theta \psi) \Big)\; dv, 
\\
\mathcal{B}^\lambda \psi : &=   -  \int \mu^+_e (1- \mathcal{Q}^+_\lambda)(\hat v_\theta \psi) \; dv 
- \int  \mu^-_e (1- \mathcal{Q}^-_\lambda)(\hat v_\theta \psi)\; dv, 
\\(\mathcal{B}^\lambda)^* \psi : &=   - \int \hat v_\theta\mu^+_e (1 - \mathcal{Q}^+_\lambda)\psi \; dv 
-   \int \hat v_\theta\mu^-_e (1 - \mathcal{Q}^-_\lambda)\psi \; dv  
.\end{aligned}                                                          \end{equation}
We also introduce 
\begin{equation}\label{def-Loperator} \mathcal{L}^\lambda: = \mathcal{A}_2^\lambda + (\mathcal{B}^\lambda)^* (\mathcal{A}_1^\lambda)^{-1} \mathcal{B}^\lambda .\end{equation}
We then have 
\begin{lemma}\label{lem-reducedVM} 
The Maxwell equations \eqref{Lap-Maxwell} are equivalent to the equations  
\begin{equation}\label{Maxwell-relations}
 \mathcal{A}_1^\lambda \varphi = \mathcal{B}^\lambda \psi ,\qquad \qquad  \mathcal{A}_2^\lambda \psi  +  (\mathcal{B}^\lambda)^* \varphi = 0.
\end{equation}
\end{lemma}
\begin{proof} 
Using  \eqref{def-f}, we  write the Poisson equation for $\varphi$ as
$$\begin{aligned} - \Delta \varphi  &= \int (f^+ - f^-) (r,v) \;dv  \\& =  \int \mu^+_e (1 - \mathcal{Q}^+_\lambda)\varphi \; dv +   \int \mu^-_e (1 - \mathcal{Q}^-_\lambda)\varphi \; dv   +  \int r(\mu^+_p + \mu^-_p) dv \psi  \\&\qquad + \int \mu^+_e \mathcal{Q}^+_\lambda(\hat v_\theta \psi)\; dv  + \int \mu^-_e \mathcal{Q}^-_\lambda(\hat v_\theta \psi)\; dv
.\end{aligned}$$
Note  that 
$\partial_{v_\theta} [\mu^\pm] = r \mu^\pm_p + \hat v_\theta \mu^\pm_e$  so that 
$$
\int r(\mu^+_p + \mu^-_p) dv  = - \int \hat v_\theta (\mu^+_e + \mu^-_e) \;dv .$$
This gives the first equation in \eqref{Maxwell-relations} by definition.  
Similarly we  write
$$\begin{aligned}  ( - \Delta_r + \lambda^2) \psi   &= \int \hat v_\theta (f^+ - f^-) (r,v) \;dv  \\& =  \int \hat v_\theta\mu^+_e (1 - \mathcal{Q}^+_\lambda)\varphi \; dv +   \int \hat v_\theta\mu^-_e (1 - \mathcal{Q}^-_\lambda)\varphi \; dv   +  \int r\hat v_\theta(\mu^+_p + \mu^-_p) dv \psi  \\&\qquad + \int \hat v_\theta \mu^+_e \mathcal{Q}^+_\lambda(\hat v_\theta \psi) \; dv + \int \hat v_\theta \mu^-_e \mathcal{Q}^-_\lambda(\hat v_\theta \psi)\; dv, 
 \end{aligned}$$
which is equivalent to the second equation in \eqref{Maxwell-relations}.
\end{proof}

As in Lemma \ref{lem-AB0property}, we now state some properties of these operators.  
We recall that the spaces $\cV$ and $\cV^\dagger$, which are defined in Section \ref{sec-spaces}, 
incorporate the boundary conditions.

\begin{lemma}\label{lem-ABlambda} 
For any $\lambda>0$,  

(i) $\mathcal{A}_1^\lambda$ is self-adjoint and positive definite on $L^2_r(\DD )$ with  domain $\mathcal{V}$. 
Moreover,  $\mathcal{A}_1^\lambda$ maps from $\mathcal{V}$  one-to-one onto 
the set $1^\perp := \{g \in L^2_r~:~ \int_\DD g \; dx =0\}$. 

(ii) $\mathcal{B}^\lambda$ is a bounded operator on $L^2_r(\DD )$ with  its adjoint operator 
$(\mathcal{B}^\lambda)^* $ defined  in \eqref{def-operators}.  
The range of $\mathcal{B}^\lambda$ is contained in $1^\perp$.  

(iii) $\mathcal{A}_2^\lambda$ and $\mathcal{L}^\lambda$ are self-adjoint on $L^2_r(\DD )$ with their 
common domain $\mathcal{V}^\dagger$. 
\end{lemma}

\begin{proof} We first check the self-adjointness of $\mathcal{A}^\lambda_j$ and the formula for $(\mathcal{B}^\lambda)^*$.  
Since $\mu^+$ is constant on trajectories and in view of \eqref{def-operators}, it clearly suffices to prove, 
for smooth functions $g$ and $h$ specular on the boundary, that
\begin{equation}\label{adjointP}
\int_\DD  \int \mu_e^+h(x,v) \mathcal{Q}^+_\lambda (g(x,v)) \; dvdx 
= \int_\DD  \int \mu_e^+g(x,\tilde v) \mathcal{Q}^+_\lambda (h(x,\tilde v)) \; dvdx,              \end{equation}
where we denote $\tilde v = (-v_r,v_\theta)$. 
In order to prove \eqref{adjointP},  we recall the definition of $\mathcal{Q}^+_\lambda$ and  use the change of variables  
 $$(y,w):=(X^+(s;x,v),V^+(s;x,v)), \qquad (x,v) := (X^+(-s;y,w),V^+(-s;y,w)),$$
which has Jacobian one where it is defined (see Lemma \ref{lem-trajectory}).  
So we can write the left side of \eqref{adjointP} as  
$$\begin{aligned}
\int_{-\infty}^0\int_\DD  \int \lambda &e^{\lambda s}\mu_e^+h(x,v)\  g(X^+(s;x,v),V^+(s;x,v)) \; dvdxds
\\ &= \int_{-\infty}^0\int_\DD  \int \lambda e^{\lambda s}\mu_e^+h(X^+(-s;y,w),V^+(-s;y,w))\  g(y,w) \; dwdyds.
\end{aligned}$$
Observe that the characteristics defined by  \eqref{trajectory} and \eqref{traj-reflection1} 
are invariant under the time-reversal transformation   
 $s\mapsto -s$, $r \mapsto r$, $v_r\mapsto -v_r$, and $v_\theta \mapsto v_\theta$.  
Thus   
$$ X^+(-s;x,v) = X^+(s;x,\tilde v), \qquad V^+(-s;x,v) = \tilde V^+ (s;x,\tilde v), $$
at least if we avoid the boundary.   Using this invariance, we obtain 
$$\begin{aligned}
\int_{-\infty}^0\int_\DD  \int \lambda &e^{\lambda s}\mu_e^+h(x,v)\ g(X^+(s;x,v),V^+(s;x,v)) \; dvdxds
\\ &= \int_{-\infty}^0\int_\DD  \int \lambda e^{\lambda s}\mu_e^+h(X^+(s;y,\tilde w),\tilde V^+(s;y,\tilde w))\ g(y,w) \; dwdyds
\\ &= \int_{-\infty}^0\int_\DD  \int \lambda e^{\lambda s}\mu_e^+h(X^+(s;x,v),\tilde V^+(s;x,v))\ g(x,\tilde v) \; dvdxds,
\end{aligned}$$
in which the last identity comes from the change of notation   $(x,v): = (y,\tilde w)$.  
By  definition of $\mathcal{Q}^+_\lambda$, this result is precisely the identity \eqref{adjointP}.   
A similar calculation holds for the $-$ case.  This proves the adjoint properties claimed in the lemma. 
\bigskip

Next we show that all the integral terms in \eqref{def-operators} are bounded operators on $L^2_r(\DD )$.  
For instance, we have 
\begin{equation}\label{est-Plambda}\begin{aligned}
\Big| \int_\DD  \int &\mu_e^+\psi \mathcal{Q}^+_\lambda\varphi \; dvdx\Big| 
 = \Big| \int_{-\infty}^0\int_\DD  \int \lambda e^{\lambda s}\mu_e^+\psi \varphi(X^+(s)) \; dvdxds\Big|      \\&
\le \Big(\int_{-\infty}^0\lambda e^{\lambda s} \int_\DD  \int |\mu_e^+||\psi|^2\; dvdxds\Big)^{1/2}   
\Big(\int_{-\infty}^0\lambda e^{\lambda s}\int_\DD  \int |\mu_e^+| |\varphi(X^+(s))|^2 \; dvdxds\Big)^{1/2}     \\&
\le \sup_\DD \Big(\int |\mu_e^+| \; dv\Big) \|\psi \|_{L^2_\DD}\|\varphi \|_{L^2_\DD} .  
 \end{aligned}
\end{equation}
In the last step we  made the change of variables $(x,v) = (X^+(s;x,v),V^+(s;x,v))$ in the integral for $\varphi$, which is possible thanks to \eqref{change-variable}.  
Similar estimates hold for the other integrals since $\hat v_\theta$ is bounded by one.  
This proves that $\mathcal{B}^\lambda$ is bounded on $L^2_r(\DD)$ and also that the integral 
terms in $\mathcal{A}_1^\lambda$ and $\mathcal{A}_2^\lambda$ are relatively compact 
with respect to $-\Delta$ and $-\Delta_r$, respectively.  
Therefore $\mathcal{A}_1^\lambda$ and $\mathcal{A}_2^\lambda$ are well-defined operators on $L^2_r(\DD )$ 
with  domains $\mathcal{V}$ and $\mathcal{V}^\dagger$, which are the same as the 
domains of $-\Delta$ and $-\Delta_r$, respectively. 

\bigskip
Taking $\psi = \varphi$ in the previous estimate, we have  
 $$
 \Big| \int_\DD  \int \mu_e^+\varphi \mathcal{Q}^+_\lambda\varphi \; dvdx\Big| 
\le \int_{-\infty}^0\lambda e^{\lambda s} \int_\DD  \int |\mu_e^+||\varphi|^2\; dvdx =  -\int_\DD  \int \mu_e^+|\varphi|^2\; dvdx $$
so that 
$$  
-\int_\DD  \int \mu^+_e \varphi \ (1 - \mathcal{Q}^+_\lambda)\varphi \; dv   
\ge 0.  $$
Thus  $\mathcal{A}_1^\lambda \ge 0$, and $\mathcal{A}_1^\lambda\varphi =0$ if and only if $\varphi$ is a constant.  
Since $\mathcal{A}_1^\lambda$ has discrete spectrum, it is invertible on the set orthogonal  to the 
kernel of  $\mathcal{A}_1^\lambda$.   That is, it is invertible on $\{g \in \mathcal{V}~:~ \int_\DD g \; dx =0\}$.  
For the invertibility of $\mathcal{A}_1^\lambda$ on the range of $\mathcal{B}^\lambda$, 
we note by \eqref{adjointP} and  $ \mathcal{Q}^\pm_\lambda(1)=1$ that   
$$
\int_\DD\int \mu_e^\pm(1- \mathcal{Q}^\pm_\lambda)(\hat v_\theta \psi) \; dvdx 
= \int_\DD\int \mu_e^\pm \hat v_\theta \psi (1- \mathcal{Q}^\pm_\lambda)(1) \; dvdx= 0.$$ 
This shows that $\mathcal{B}^\lambda\psi \in \{g \in \mathcal{V}~:~ \int_\DD g \; dx =0\}$  for all $\psi$.   
Thus $(\mathcal{A}_1^\lambda)^{-1}$ is well-defined on the range of $\mathcal{B}^\lambda$, 
and so $\mathcal{L}^\lambda$ is well-defined. 
The self-adjoint property of $\mathcal{L}^\lambda$ is clear from 
that of $\mathcal{A}_1^\lambda$. 
\end{proof}

Part {\it (i)} of Lemma \ref{lem-ABlambda} in particular shows that for each $\psi \in L^2_r(\DD)$ 
there exists a unique radial function $\varphi \in H^2_r(\DD)$ that solves  
\begin{equation}\label{phi-def}
\mathcal{A}_1^\lambda \varphi = \mathcal{B}^\lambda \psi, \qquad \int_\DD \varphi \; dx =0, \qquad \partial_r \varphi(1) =0.\end{equation}

\subsection{Construction of a growing mode}\label{ss-constr}
\begin{lemma} \label{growing psi}
If $\mathcal{L}^0 \not \ge 0$, then there exist a $\lambda>0$ and a nonzero function $\psi \in H^{2\dagger}(\DD)$ such that $\mathcal{L}^\lambda\psi =0$ and $\psi$ satisfies the Dirichlet condition on the boundary. 
\end{lemma}

\begin{proof} The proof is similar to the one and a half dimensional case given in \cite{LS1} 
so that we merely outline the main steps as follows.  

\noindent  {\it (i)} $\mathcal{L}^\lambda \ge 0$ for large $\lambda$. 
 
\noindent {\it (ii)}  For all $\psi \in L^2_r$, $\mathcal{L}^\lambda \psi$ converges strongly to $\mathcal{L}^0\psi$ in $L^2$ 
as $\lambda \to 0$, and thus $\mathcal{L}^\lambda \not \ge 0$ when $\lambda$ is small.  

\noindent {\it (iii)} The smallest eigenvalue 
$ \kappa^\lambda: = \inf _{\psi} ~\langle \mathcal{L}^\lambda \psi, \psi \rangle_{L^2}$ 
 of $\mathcal{L}^\lambda$ is continuous in $\lambda>0$, 
where the infimum is taken over $\psi \in \mathcal{V}^\dagger$ with $\|\psi\|_{L^2} =1$. 

These three steps imply that $\kappa^\lambda$ must be zero for some $\lambda>0$, from which the lemma follows. 
To prove  {\it (i)}, it is easy to see that $\mathcal{L}^\lambda$ is nonnegative for large $\lambda$, 
since $(\mathcal{B}^\lambda)^* (\mathcal{A}^\lambda_1)^{-1} \mathcal{B}^\lambda \ge 0$ 
and $\langle \mathcal{A}_2^\lambda \psi,\psi\rangle_{L^2}$ is sufficiently large when $\lambda$ is large. 
As for {\it(ii)}, to show the convergence of $\mathcal{L}^\lambda$ to $\mathcal{L}^0$ as $\lambda \to 0$, 
we use the remarkable fact, proved in \cite[Lemma 2.6]{LS1}, that for all $g \in \cH$  the strong limit  
\begin{equation}\label{conv-proj} 
\lim_{\lambda \to 0^+} \mathcal{Q}^\pm_\lambda g = \mathcal{P}^\pm g\end{equation}
is valid in the $L^2_{|\mu_e^\pm|}=\cH$ norm. 
Here the  $\mathcal{P}^\pm$ are the orthogonal projections of $L^2_{|\mu_e^\pm|}$ onto the kernels of $\oD ^\pm$. 
However, it should be noted that  the convergence is not true in the operator norm.  
For all $\psi \in L^2_r$, we use \eqref{conv-proj} and write 
$$ \mathcal{A}_2^\lambda \psi - \mathcal{A}_2^0\psi 
= \lambda^2 \psi - \sum _\pm \int \hat v_\theta \mu_e^\pm \Big[\mathcal{Q}^\pm_\lambda(\hat v_\theta \psi) 
- \mathcal{P}^\pm(\hat v_\theta \psi)\Big]\; dv,$$
thereby obtaining  the convergence of $\mathcal{A}_2^\lambda \psi$ to  $\mathcal{A}_2^0\psi $ in $L^2$. 
Similarly, $\mathcal{A}_1^\lambda$ and $\mathcal{B}^\lambda$ also converge strongly in $L^2$ to  $\mathcal{A}_1^0$ 
and $\mathcal{B}^0$, respectively, and so does $\mathcal{L}^\lambda$ to $\mathcal{L}^0$. 
Finally, estimates  similar to \eqref{est-Plambda}
yield
$$\begin{aligned}
\Big| \int_\DD  \int \hat v_\theta\mu_e^+\psi &( \mathcal{Q}^+_\lambda(\hat v_\theta \psi) - \mathcal{P}^+_\mu(\hat v_\theta \psi)) \; dvdx\Big|  
\\&= \Big| \int_{-\infty}^0\int_\DD  \int \Big(\lambda e^{\lambda s} - \mu e^{\mu s}\Big)\hat v_\theta \mu_e^+\psi \hat V_\theta(s) \psi(X^+(s)) \; dvdxds\Big|
\\&\le  \int_{-\infty}^0 |\lambda e^{\lambda s} - \mu e^{\mu s}|\Big( \int_\DD  \int |\mu_e^+||\psi(x)|^2\; dvdx \Big)^{1/2} \Big( \int_\DD \int |\mu^+_e ||\psi(X^+(s)) |^2\; dvdx\Big)^{1/2} ds
\\&
\le C_0\Big(\int_{-\infty}^0|\lambda e^{\lambda s} - \mu e^{\mu s}|\; ds \Big) \|\psi \|_{L^2_\DD}^2
\le C_0 |\log \lambda - \log \mu|\|\psi \|_{L^2_\DD}^2
 ,\end{aligned}     $$
and thus 
$$ 
\langle \mathcal{A}_2^\lambda \psi - \mathcal{A}_2^\mu\psi, \psi \rangle    
\le C_0 \Big( |\lambda - \mu| + |\log \lambda - \log \mu|\Big)\|\psi \|_{L^2_\DD}^2       $$ 
for all $\lambda,\mu>0$ and $\psi \in L^2_r$. Similarly, we obtain the same estimate for $\mathcal{L}^\lambda$, which proves the continuity of the lowest eigenvalue $\kappa^\lambda$ of $\mathcal{L}^\lambda$.   
\end{proof}

Using $\psi$, we can now construct the growing mode.  
\begin{lemma}\label{lem-growingmode} 
Let  $\lambda,\psi$ be as in Lemma \ref{growing psi},  let $\varphi$ be as in \eqref{phi-def}, 
and let $f^\pm$ be defined by \eqref{def-f}.  
Then $(e^{\lambda t}f^\pm,e^{\lambda t}\varphi,e^{\lambda t}\psi)$ is a growing mode of the linearized Vlasov-Maxwell 
system. 
\end{lemma}

\begin{proof}  
Because $\mathcal L^\lambda\psi=0$ and due to the definition of $\varphi$,  both parts of \eqref{Maxwell-relations} 
are satisfied.  Therefore \eqref{Lap-Maxwell}  is satisfied.  
These are the first and third Maxwell equations in \eqref{Maxwell-eqs} 
together with the boundary conditions for $\varphi$ and $\psi$. 
Next, the specular boundary condition for $f^\pm$ follows directly by definition \eqref{def-f} 
and the fact that $\mathcal{Q}^\lambda(g)$ is specular on the boundary if $g$ is.  
It remains to check the Vlasov equations and the middle Maxwell equation in \eqref{Maxwell-eqs}, namely   
 $\lambda\partial_r \varphi = j_r  $.

We begin with the equation for $f^+$. 
Recall that $  (X^+(t;x,v),V^+(t;x,v))$
 is the particle trajectory  initiating from $(x,v)$. Evaluating \eqref{def-f} along the trajectory, we have  
$$\begin{aligned}
 f^+(X^+(t),V^+(t))  &=  \mu^+_e \varphi (X^+(t))  +  \mu^+_p R^+(t)\psi(X^+(t))  
 - \mu^+_e \int_{-\infty}^0 \lambda e^{\lambda s} \varphi (X^+(s; X^+(t),V^+(t))) \; ds \\&\qquad 
 +  \mu^+_e \int_{-\infty}^0 \lambda e^{\lambda s} \hat V^+_\theta (s;X^+(t),V^+(t))\psi (X^+(s;X^+(t),V^+(t))) \; ds.
 \end{aligned}$$
By the group property $(X^+(s; X^+(t),V^+(t)) = X^+(s+t)$ and $V^+(s; X^+(t),V^+(t)) = V^+(s+t),$ together with integration by parts in $s$, 
we have for each $t$
$$\begin{aligned}
 f^+(X^+(t),V^+(t))  &=  \mu^+_e \varphi (X^+(t)) +  \mu^+_p R^+(t)\psi(X^+(t))  \\&\qquad 
 - \mu^+_e \int_{-\infty}^0 \lambda e^{\lambda s} \Big[\varphi (X^+(s+t))  - \hat V^+_\theta (s+t)\psi (X^+(s+t)) \Big] \; ds
 \\
  &=   \mu^+_p R^+(t)\psi(X^+(t))  + \mu^+_e e^{-\lambda t} \int_{-\infty}^t e^{\lambda s} \Big[\D_s \varphi (X^+(s) ) 
  + \lambda \hat V^+_\theta (s)\psi (X^+(s)) \Big] \; ds.
  \end{aligned}$$
Differentiation of this identity yields 
$$\begin{aligned}
\frac{d}{dt} \Big( e^{\lambda t}f^+(X^+(t),V^+(t))\Big)&=   \mu^+_p \frac{d}{dt}\Big(e^{\lambda t}R^+(t)\psi(X^+(t))\Big) + \mu^+_e e^{\lambda t} \Big[\D_t \varphi (X^+(t) ) + \lambda \hat V^+_\theta (t)\psi (X^+(t)) \Big] .
  \end{aligned}$$
  We evaluate the above identity for $t\in (0,\epsilon)$ and let $\epsilon \to 0$.  
Note that by Lemma \ref{lem-trajectory} the functions $f^+(X^+(t),V^+(t))$, $\varphi(X^+(t))$ and $\psi(X^+(t))$ 
are piecewise $C^1$ smooth.   
Using the evolution \eqref{trajectory} and \eqref{traj-reflection1}, we obtain 
$$\begin{aligned}
\lambda f^+ + \oD ^+f^+  
=  \mu^+_e \hat v_r\D_r \varphi + r\mu^+_p  \hat v_r\D_r \psi + \mu^+_p \hat v_r \psi + \lambda( \mu^+_e \hat v_\theta + r\mu^+_p) \psi  .  
  \end{aligned}$$
This is the Vlasov equation \eqref{lin-VM} for $f^+$. A similar verification can be done for $f^-$. 

\bigskip

Finally, we verify the remaining Maxwell equation   $\lambda\partial_r \varphi = j_r  $.  
Indeed, by performing the integration in $v$ of the Vlasov equations \eqref{linearization}, 
we easily obtain $\lambda \rho  + \nabla \cdot \vj =0$. 
Together with the Poisson equation in \eqref{Lap-Maxwell}, this yields 
$$ - \Big(\partial_r + \frac 1r \Big) (\lambda \partial_r \varphi)  = -\lambda \Delta \varphi = \lambda \rho = - \Big(\partial_r + \frac 1r \Big) j_r$$
Thus $r(\lambda \partial_r \varphi - j_r)$ must be a constant.     However, at the boundary $r=1$  
we have $\partial_r \varphi=0$ and $j_r=0$ by the specular boundary condition on $f^\pm$. 
So $\lambda \partial_r \varphi - j_r=0$.  
\end{proof}
This completes the proof of Theorem \ref{theo-main}.

\section{Examples} \label{sec-examples}
The purpose of this section is to exhibit some explicit examples of stable and unstable equilibria, 
and thereby prove Theorem \ref{theo-examples}. 

%
%

\subsection{Stable examples}
By Theorem \ref{theo-main} the  condition for spectral stability is 
\begin{equation}\label{stabcond-L0}\mathcal{L}^0  
= \mathcal{A}_2^0 + (\mathcal{B}^0)^* (\mathcal{A}_1^0)^{-1} \mathcal{B}^0 \ge 0.\end{equation}
For each $\psi$ in the domain of $\mathcal{L}^0$ (thus in particular satisfying the Dirichlet boundary condition), we have 
$$
\langle \mathcal{L}^0 \psi,\psi\rangle_{L^2} = \langle \mathcal{A}_2^0 \psi,\psi\rangle_{L^2} + \langle \mathcal{A}_1^0\varphi,\varphi\rangle_{L^2} ,
$$
where $\varphi$ solves $\mathcal{A}_1^0\varphi = B^0\psi$ with the Neumann boundary condition on $\varphi$. 
Recall that 
$$\begin{aligned}  
\mathcal{A}_1^0 \varphi & = - \Delta \varphi  - \int \mu^+_e(1-\mathcal{P}^+) \varphi \; dv 
- \int \mu^-_e (1-\mathcal{P}^-) \varphi \; dv,  \\
 \mathcal{A}_2^0 \psi & = \Big( - \Delta + \frac {1}{r^2}\Big)  \psi   
 -  \int r\hat v_\theta (\mu_p^+ + \mu_p^-) \; dv \psi  
 -   \int \hat v_\theta\Big(\mu_e^+ \mathcal{P}^+(\hat v_\theta \psi) +\mu_e^- \mathcal{P}^-(\hat v_\theta \psi) \Big)\; dv, 
\end{aligned}$$
in which $\mu^\pm$ denote $\mu^\pm(e^\pm,p^\pm) = \mu^\pm(\langle v \rangle  \pm \varphi^0,r(v_\theta \pm \psi^0))$. 
Integration by parts, together with the boundary conditions for $\varphi$ and $\psi$,  
and the orthogonality of $\mathcal{P}^\pm$ and $1-\mathcal{P}^\pm$ 
(separately for $+$ and $-$)  lead to the expressions 
\begin{equation}\label{L2norm-A012}\begin{aligned}
\langle \mathcal{A}_1^0\varphi,\varphi\rangle_{L^2} 
 &= \int_\DD|\D_r \varphi|^2\; dx  -  \int_\DD\int \mu_e^+ |(1-\mathcal{P}^+)(\varphi)|^2 \; dvdx -  \int_\DD\int \mu_e^- |(1-\mathcal{P}^-)(\varphi)|^2 \; dvdx,\\
 \langle \mathcal{A}_2^0 \psi,\psi\rangle_{L^2} &=\int_\DD \Big(|\D_r\psi|^2 + \frac 1{r^2}|\psi|^2 \Big) \; dx  -  \int _\DD \Big( \int r \hat v_\theta (\mu_p^+ + \mu_p^-) \; dv\Big)  |\psi|^2\; dx  
 \\&\quad-  \int_\DD\int \Big[\mu_e^+ |\mathcal{P}^+(\hat v_\theta \psi)|^2 + \mu_e^- |\mathcal{P}^-(\hat v_\theta \psi)|^2\Big] \; dvdx .
\end{aligned}\end{equation}
Due to 
 the assumption $\mu_e^\pm < 0$, it is clear that $\langle \mathcal{A}_1^0\varphi,\varphi\rangle_{L^2} \ge 0$ and so are the first and last terms in $ \langle \mathcal{A}_2^0 \psi,\psi\rangle_{L^2}$.  
We now  exhibit two explicit sufficient conditions for  $ \langle \mathcal{A}_2^0 \psi,\psi\rangle_{L^2}$  to be nonnegative.  
 This is Theorem \ref{theo-examples} {\it(i)} and {\it (ii)}.   
\begin{theorem}\label{lem-stab} 
Let $(\mu^\pm,\varphi^0,\psi^0)$ be an inhomogenous equilibrium.  

(i) If 
\begin{equation}\label{suff-stab}
p \mu^\pm_p(e,p)\le 0, \qquad \quad \forall ~ e,p, 
 \end{equation} 
then the equilibrium is spectrally stable provided that $\varphi^0\in L^\infty$  and $\psi^0$ is sufficiently small in $L^\infty$. 

(ii) If 
\begin{equation}\label {suff-stab2}
|\mu_p^\pm (e,p)|  \le  \frac \epsilon {1+|e|^\gamma} , \end{equation}
for some $\gamma>2$, with $\epsilon$ sufficiently small and $\varphi^0=0$ but $\psi^0$  not necessarily small, 
then the equilibrium is spectrally stable.  

\end{theorem}
\begin{proof} First consider case {\it (i)}.  
We only need to show that $\mathcal{A}_2^0 \ge 0$. 
Let us look at the second integral  of $ \langle \mathcal{A}_2^0 \psi,\psi\rangle_{L^2}$ in \eqref{L2norm-A012}.   
By the definition \eqref{ep} of $p^\pm$, we may write  
$$\begin{aligned}\int r \hat v_\theta \mu_p^\pm(e^\pm,p^\pm) \; dv  
= \int \langle v \rangle^{-1}p^\pm \mu_p^\pm(e^\pm,p^\pm) \; dv  
\mp r\psi^0 \int \langle v \rangle^{-1}\mu_p^\pm(e^\pm,p^\pm) \; dv , 
\end{aligned}$$
in which the first term on the right is nonpositive due to \eqref{suff-stab}.  Therefore we have 
$$\begin{aligned}
 -  \int _\DD \Big( \int r \hat v_\theta (\mu_p^+ + \mu_p^-) \; dv\Big)  |\psi|^2\; dx & \ge   \int_{\DD}  \int \langle v \rangle^{-1} (\mu_p^+ - \mu_p^-) \; dv\; r \psi^0|\psi|^2 \; dx
\\
& \ge  -  \sup _r |\psi^0|\Big( \sup_r \int \langle v \rangle^{-1} (|\mu_p^+|+ |\mu_p^-| ) \; dv \Big)\int_{\DD}  r |\psi|^2 \; dx 
.\end{aligned}
 $$
Now by the Poincar\'e inequality, 
$$\int_{\DD}  r |\psi|^2 \; dx \le c_0\int_\DD \Big(|\D_r\psi|^2 + \frac 1{r^2}|\psi|^2 \Big) \; dx,$$
for some constant $c_0$. In addition, thanks to  assumption \eqref{mu-cond}, the supremum over $r\in [0,1]$ of 
$\int \langle v \rangle^{-1}(|\mu_p^+|+ |\mu_p^-| ) \; dv$ is finite if $\varphi^0$ is bounded. 
Thus if the sup norm of $\psi^0$ is sufficiently small, or more precisely if $\psi^0$ satisfies
\begin{equation}\label{sup-psibound} c_0 \sup _r |\psi^0|\Big( \sup_r \int \langle v \rangle^{-1}(|\mu_p^+|+ |\mu_p^-| ) \; dv \Big) \le 1,\end{equation}
then the second term in $\langle \mathcal{A}_2^0 \psi,\psi\rangle_{L^2}$ is smaller than the first, and so the operator $\mathcal{A}_2^0$ is nonnegative. 

Case {\it (ii)} is even easier.  As above, we only have to bound the second term in 
$ \langle \mathcal{A}_2^0 \psi,\psi\rangle_{L^2}$.
Using $|r\hat v_\theta| \le 1$  and $e=\langle v\rangle$ together with \eqref{suff-stab2}, we have 
$$
\Big | \int _\DD \Big( \int r \hat v_\theta (\mu_p^+ + \mu_p^-) \; dv\Big)  |\psi|^2\; dx\Big|  
 \le \iint \frac\epsilon{1+|v|^\gamma} dv |\psi|^2 dx  \le  C\epsilon \int|\psi|^2 dx. $$
If $\epsilon$ is sufficiently small, the second term is smaller than the positive terms.  
\end{proof}

\subsection{Unstable examples}
For instability, it suffices to find a single function in the domain of $\mathcal{L}^0$ such that 
$\langle \mathcal{L}^0 \psi,\psi\rangle_{L^2} <0$. We shall construct some examples where this is the case. 
We limit ourselves to a purely magnetic equilibrium $(\mu^\pm,\vE^0,B^0)$ with 
$\vE^0=0$ and $B^0 = \frac 1r \partial_r (r \psi^0)$.  Thus $e=\langle v \rangle$ and $p^\pm = r(v_\theta \pm \psi^0)$.  
%
%
%
 In this subsection, we shall also make the  assumption  that 
\begin{equation}\label{simplified-cond} 
\mu^+(e ,p) = \mu^-(e ,-p),\qquad \forall e,p.\end{equation}
This assumption holds for example if $\mu^+ = \mu^-$ is an even function of $p$. 
It greatly simplifies the verification of the spectral condition on $\mathcal{L}^0$.
%
%

We now show that  assumption \eqref{simplified-cond} implies that the operator $\mathcal{B}^0$ vanishes 
and so $\mathcal{L}^0$ simply reduces to $\mathcal{A}_2^0$. Indeed, let us recall that 
$$
\mathcal{B}^0 \psi  = r \int \Big(\mu_p^+(e ,p^+) +  \mu_p^-(e ,p^-)\Big )\; dv \psi  
+ \int \Big(\mu_e^+(e ,p^+) \mathcal{P}^+(\hat v_\theta \psi) 
+\mu_e^-(e ,p^-) \mathcal{P}^-(\hat v_\theta \psi) \Big) \; dv ,  $$
in which $e = \langle v \rangle$ and $p^\pm = r(v_\theta \pm \psi^0)$. 
For the first term in $\mathcal{B}^0 \psi$, we again note by  \eqref{simplified-cond}    that the function 
$$\mu_p^+ (e ,p^+)+  \mu_p^-(e ,p^-) = - \mu_p^-(\langle v\rangle , - r(v_\theta + \psi^0)) + \mu_p^-(\langle v\rangle ,r(v_\theta - \psi^0))$$
is odd in $v_\theta$. Thus, the first integral in $\mathcal{B}^0$ vanishes. 
As for the second integral, we note that 
$$\oD ^-= \hat v_r \D_r   + \Big( (- \hat v_\theta) B^0 +  \frac 1r (-v_\theta)(- \hat v_\theta)\Big)\D_{v_r} 
 - \Big(\hat v_r B^0  + \frac 1r v_r (- \hat v_\theta)\Big)   \D_{-v_\theta}  .$$
That is, $\oD ^-$ acting on functions $f(v_r,-v_\theta)$ is the same as $\oD ^+$ acting on $f(v_r,v_\theta)$ 
(ignoring the dependence on $x$).   
As a consequence we have 
$$\mathcal{P}^+ (f(v_r,v_\theta)) (v_r,v_\theta) = \mathcal{P}^- (f(v_r,-v_\theta)) (v_r,-v_\theta).$$  
Using this identity together with the fact that $\mathcal{P}^\pm (-f) = -\mathcal{P}^\pm (f)$, we have  
$$\begin{aligned}
&\mu_e^+(e,r(v_\theta+\psi^0)\mathcal{P}^+(\hat v_\theta \psi) (v_\theta)  
+\mu_e^-(e,r(v_\theta-\psi^0) \mathcal{P}^-(\hat v_\theta \psi) (v_\theta)
\\= &- \mu_e^- (e ,r(-v_\theta - \psi^0))\mathcal{P}^-(\hat v_\theta \psi) (-v_\theta) 
-\mu_e^+ (e ,r(-v_\theta + \psi^0))\mathcal{P}^+(\hat v_\theta \psi)(-v_\theta)
\\= &-\mu_e^+ (e ,r(-v_\theta + \psi^0))\mathcal{P}^+(\hat v_\theta \psi)(-v_\theta) - \mu_e^- (e ,r(-v_\theta - \psi^0))\mathcal{P}^-(\hat v_\theta \psi) (-v_\theta) 
. 
\end{aligned}$$
Thus the function 
$\mu_e^+\mathcal{P}^+(\hat v_\theta \psi) +\mu_e^- \mathcal{P}^-(\hat v_\theta \psi)$ 
is odd  in $v_\theta$, 
so that the second integral in $\mathcal{B}^0\psi $ vanishes. 
Similarly, we easily obtain
\begin{equation}\label{new-A02} 
 \mathcal{A}_2^0 \psi  = -\Delta_r \psi   - 2  \int r \hat v_\theta \mu_p^- (e,p^-)\; dv \psi 
  -   2\int \hat v_\theta \mu_e^-(e,p^-)  \mathcal{P}^-(\hat v_\theta \psi) \; dv.
\end{equation}
We summarize the above considerations in the following lemma.  
\begin{lemma}\label{lem-simplified} 
If $(\mu^\pm,0,\psi^0)$ be an equilibrium satisfying \eqref{simplified-cond}, then $\mathcal{B}^0 =0$ and 
$$\mathcal{L}^0 = \mathcal{A}^0_2,$$
for $\mathcal{A}^0_2$ as in \eqref{new-A02}. 
In particular, $\mathcal{A}_2^0\not \ge 0$ implies the spectral instability of $(\mu^\pm,0,\psi^0)$. 
\end{lemma} \bigskip 

\subsubsection{Homogeneous equilibria}
We start with the homogenous case  $\vE^0=0$ and $B^0=0$, in which case the linearized Vlasov operator 
reduces to $\oD  = \hat v \cdot \nabla_x = \hat v_r\partial_r$. 
The projection $\mathcal{P} = \mathcal{P}^\pm$ is simply the average 
$$ \mathcal{P}(\psi) = \frac 1\pi \int_\DD \psi(r) \; dx  = 2 \int_0^1 \psi(r)\,r\,dr$$
for any radial function $\psi = \psi(r)$.  In addition,  noting that $\oD (r \hat v_\theta) =0$, we have 
$$ \mathcal{P}(\hat v_\theta \psi) =  r\hat v_\theta \mathcal{P}(\frac \psi r)  
= r \hat v_\theta  \frac 1\pi \int_\DD r^{-1}\psi(r) \; dx  = 2 r \hat v_\theta \psi_R, \qquad \psi_R: =  \int_0^1 \psi(r) \; dr. $$
Thus, as  in \eqref{L2norm-A012}, we obtain the basic identity 
$$
\begin{aligned}
 \langle \mathcal{A}_2^0 \psi,\psi\rangle_{L^2} 
 &=\int_\DD\Big(|\D_r\psi|^2 + \frac 1 {r^2} |\psi|^2 \Big) \; dx  
 -  2\int _\DD \Big( \int r \hat v_\theta \mu_p^- \; dv\Big)  |\psi|^2\; dx  
 -  2 \int_\DD\int  \mu_e^- |\mathcal{P}(\hat v_\theta \psi)|^2 \; dvdx \\
 &=\int_\DD\Big(|\D_r\psi|^2 + \frac 1 {r^2} |\psi|^2 \Big) \; dx  
 -  2\int _\DD \Big( \int \langle v \rangle^{-1}p \mu_p^- \; dv\Big)  |\psi|^2\; dx  
 -  8 \int_\DD\int  r^2 \hat v^2_\theta \mu_e^-\; dvdx | \psi_R|^2 \\ 
 &= I+II+III, 
\end{aligned}$$ in which $p=r v_\theta$. 
%
%

\begin{theorem} \label{lem-mg-unstab} 
Let $\mu^\pm = \mu^\pm(e,p)$ be an homogenous equilibrium satisfying \eqref{simplified-cond}. Assume also that 
\begin{equation}\label{mg-unstab}
p \mu^-_p(e,p)\ge c_0 p^2   \nu(e), \qquad \quad \forall ~ e,p,
 \end{equation} 
for some positive constant $c_0$ and some nonnegative continuous function $\nu(e)$ such that $\nu \not \equiv 0$. 
Then there exists a positive number $K_0$ such that both of the rescaled homogenous equilibria 
(i) $\mu^{(K),\pm}(e,p): = K \mu^\pm(e, p)$ and (ii) $\mu^{(K),\pm}(e,p): = \mu^\pm(e, K p)$ 
are spectrally unstable, for all $K\ge K_0$. 
\end{theorem}

\begin{proof}   
Note that terms I and III are nonnegative since $\mu_e<0$.  So by Lemma \ref{lem-simplified}, it suffices for instability
that the middle integral II be negative and  dominate the other two.  
We begin with case {\it (i)} of the theorem.  
Observe that since $\mu^{(K),\pm}(e,p)$ satisfy \eqref{simplified-cond}, the pair $(\mu^{(K),+},\mu^{(K),-})$ 
is indeed an homogeneous equilibrium. It suffices  to construct  a function $\psi_* \in H^{2\dagger}(\DD)$ 
such that $ \langle \mathcal{A}_2^0 \psi_*,\psi_*\rangle_{L^2} <0$.   We choose  
\begin{equation}\label{def-psi1} \psi_*(r): = \left\{\begin{aligned} &\psi_1(r), \qquad \quad &0\le r\le \frac 12
\\
- &\psi_1(1-r), \qquad &\frac 12\le r\le 1,
\end{aligned}
\right.
\end{equation}
where $\psi_1(r) = \gamma_0 r(\frac12-r)$, for some normalizing constant $\gamma_0$.   
Clearly $(\psi_*)_R = 0$ so that $III=0$.    Moreover,  
$$I = \int_\DD\Big(|\D_r\psi_*|^2 + \frac 1 {r^2} |\psi_*|^2 \Big) \; dx  
= 2\pi \int_0^{1/2}\Big(|\D_r\psi_1|^2 + \frac 1 {r(1-r)} |\psi_1|^2 \Big) \; dr = 1 $$
by choice of the constant $\gamma_0$. 
So it remains to show that $II<-1$.  

 Using the assumption \eqref{mg-unstab} and the first scaling {\it (i)}, we have 
$$\begin{aligned} 
-II = 2\int _\DD \Big( \int \langle v \rangle^{-1}p \mu_p^{(K),-} \; dv\Big)  |\psi_*|^2\; dx \quad
&\ge\quad  4c_0K \pi \int_0^1 \int \langle v \rangle^{-1}p^2   \nu(e) \; dv |\psi_*|^2\;rdr 
\\
\quad&\ge\quad  4c_0K \pi \Big(\int \langle v \rangle^{-1} v_\theta^2   \nu(e) \; dv\Big) 
 \int_0^{1} r^3 |\psi_*|^2\;dr 
.\end{aligned}$$
The integral in $v$ is a finite positive constant thanks to the decay assumption on $\mu$.   
So we can choose $K$ large enough that $-II > 1$.   This settles case {\it (i)}.   

Similarly, for case {\it (ii)} we have  
$$\begin{aligned} 
-II &=
 2\int _\DD \Big( \int \langle v \rangle^{-1}(Kp) \mu_p^{-}(e,Kp) \; dv\Big)  |\psi_*|^2\; dx  \\
&\ge  4c_0K^2 \pi \int_0^1 \int \langle v \rangle^{-1}p^2   \nu(e) \; dv |\psi_*|^2\;rdr 
=  4c_0K^2 \pi \Big(\int \langle v \rangle^{-1} v_\theta^2   \nu(e) \; dv\Big) 
 \int_0^{1} r^3 |\psi_*|^2\;dr 
,\end{aligned}$$
which is again greater than one for sufficiently large $K$.  This settles case {\it (ii)}.  
\end{proof}

We remark that the constant $K_0$ in Lemma \ref{lem-mg-unstab} is certainly not optimal.  
For instance, we could take $\psi_*$ to be the ground state of the operator $-\Delta_r$, 
which is a Bessel function.     \bigskip 


\subsubsection{Inhomogeneous equilibria}
For spatially dependent equilibria we will prove a similar result. We first observe  as in the homogeneous case that the projection $ \mathcal{P}^-$ satisfies 
\begin{equation}\label{Pr-psi} \mathcal{P}^-(\psi) = \frac 1\pi \int_\DD \psi(r) \; dx  = 2 (r\psi)_R,\end{equation}
for functions $\psi = \psi(r)$, where $(\psi)_R: = \int_0^1 \psi(r)\; dr$. 
We recall from \eqref {L2norm-A012} that 
$$\begin{aligned}
 \langle \mathcal{A}_2^0 \psi,\psi\rangle_{L^2} &=\int_\DD\Big(|\D_r\psi|^2 + \frac 1 {r^2} |\psi|^2 \Big)  -  2\int _\DD \Big( \int r \hat v_\theta \mu_p^- \; dv\Big)  |\psi|^2\; dx  -  2 \int_\DD\int  \mu_e^- |\mathcal{P}^-(\hat v_\theta \psi)|^2 \; dvdx 
 \end{aligned}$$ 
$=I+II+III$, where clearly $I\ge0$ and $III\ge0$.  
Of course, both $e=\langle v \rangle$ and $p^- = r(v_\theta - \psi^0)$ belong to the kernel of $\oD ^-$.   
Thus we have 
$$
\mathcal{P}^-(\hat v_\theta \psi) = \langle v \rangle^{-1}\mathcal{P}^-(v_\theta \psi) 
= \langle v \rangle^{-1}\mathcal{P}^-\Big(r(v_\theta-\psi^0) \frac{\psi}{r} + \psi^0\psi\Big) 
=  \langle v \rangle^{-1}p^- \mathcal{P}^-\Big(\frac{\psi}{r}\Big) + \langle v \rangle^{-1}\mathcal{P}^-(\psi^0\psi) .$$
Since $\psi^0$ and $\psi$ are functions depending only on $r$, we apply \eqref{Pr-psi} to give  
\begin{equation}\label{comp-proj}
\mathcal{P}^-(\hat v_\theta \psi) = 2\langle v \rangle^{-1}p^- \psi_R + 2\langle v \rangle^{-1}(r\psi^0\psi)_R .
\end{equation}
By definition of $p^-$, we can write 
$$ \int r \hat v_\theta \mu_p^- \; dv = \int \langle v \rangle^{-1} p^- \mu_p^- \; dv 
+ r\psi^0 \int  \langle v \rangle^{-1}\mu_p^- \; dv.$$
Using the inequality $(a+b)^2\le 2 a^2 + 2b^2$, we then obtain 
\begin{equation}\label{boundA2-mg}\begin{aligned}
 \langle \mathcal{A}_2^0 \psi,\psi\rangle_{L^2} &\le \int_\DD\Big(|\D_r\psi|^2 + \frac 1 {r^2} |\psi|^2 \Big) \; dx  
 -  2\int _\DD \Big(\int \langle v \rangle^{-1}p^- \mu_p^- \; dv \Big)  |\psi|^2\; dx  
 \\&\quad-  16 \int_\DD\int  \langle v \rangle^{-2}(p^-)^2\mu_e^- \; dvdx |\psi_R|^2 
 + 2\sup_{r\in [0,1]} \Big( \int  \langle v \rangle^{-1}|\mu_p^-| \; dv\Big)\|\sqrt {r|\psi^0|}\psi\|^2_{L^2(\DD)} \\&\quad 
 + 16  \sup_{r\in [0,1]} \Big( \int  \langle v \rangle^{-2}|\mu_e^-| \; dv\Big) (r\psi^0\psi)_R^2 
= I + IIA + IIIA + IIB + IIIB.  
 \end{aligned} \end{equation}

We now scale in the variable $p$ to get the following result.  
\begin{theorem}\label{lem-unstab-inhomo} Assume that $\mu^\pm$ satisfy \eqref{simplified-cond} and that
\begin{equation}\label{mg-unstab-in}
p \mu^-_p(e,p)\ge c_0 p^2   \nu(e), \qquad \quad \forall ~ e,p,
 \end{equation} 
for some positive constant $c_0$ and some nonnegative function $\nu(e)$ such that $\nu \not \equiv 0$. 
Define $\mu^{(K),\pm}(e,p): = \mu^\pm(e, K p)$ and let $\psi^{(K),0}$ be the solution of the equation 
 \begin{equation}\label{equilibrium-psi} \begin{aligned}
- \Delta_r\psi^{(K),0} &=  \int  \hat v_\theta\Big[\mu^{(K),+}(\langle v \rangle ,r(v_\theta + \psi^{(K),0})) - \mu^{(K),-}(\langle v \rangle ,r (v_\theta - \psi^{(K),0})) \Big]\;dv,
\end{aligned}
\end{equation}
with $\psi^{(K),0} =0$ on the boundary $\D\DD$. 
Then there exists a positive number $K_0$ such that the inhomogenous purely magnetic equilibria $(\mu^{(K),\pm},0,B^{(K),0})$, with $B^{(K),0} = \frac 1r \D_r (r\psi^{(K),0})$, are spectrally unstable for all $K\ge K_0$.
\end{theorem}

\begin{proof}  As before, we will check the instability by showing $ \langle \mathcal{A}_2^0 \psi_*,\psi_*\rangle_{L^2}<0$ 
for some $\psi_*$.   As in Lemma \ref{lem-mg-unstab}, we make the simple choice of  $\psi_*$ 
given by \eqref{def-psi1} so that $(\psi_*)_R =0$, whence $IIIA=0$, with the constant chosen so that the first integral term $I$ in  
 \eqref{boundA2-mg} equals 1.  
  Thus we obtain 
\begin{equation}\label{boundA2-mg-01}\begin{aligned}
 \langle \mathcal{A}_2^0 \psi_*,\psi_*\rangle_{L^2} &\le 1  -  2\int _\DD \Big(\int \langle v \rangle^{-1} Kp^- \mu_p^-(\langle v \rangle,Kp^-) \; dv \Big)  |\psi_*|^2\; dx  
 \\&\quad+ C_0K\|\psi^{(K),0}\|_{L^\infty} 
 \sup_{r\in [0,1]} \Big( \int  \langle v \rangle^{-1}|\mu_p^-(\langle v \rangle,Kp^-)| \; dv\Big)
 \\&\quad+ C_0\|\psi^{(K),0}\|_{L^\infty} ^2 
 \sup_{r\in [0,1]} \Big( \int  \langle v \rangle^{-2}|\mu_e^-(\langle v \rangle,Kp^-)| \; dv\Big),  
 \end{aligned}
 \end{equation}
for some constant $C_0$ that depends only on the $L^2$ norm of $\psi_*$.   
We shall show that the second integral $IIA$ in this estimate dominates all the other terms if $K$ is large.  
 From the decay assumption \eqref{mu-cond} on $\mu_e^\pm$ and $\mu_p^\pm$, we have 
$$ 
\int  \langle v \rangle^{-1}|\mu_p^-(\langle v \rangle,Kp^-)| \; dv 
\le C_\mu  \int \frac{1}{\langle v \rangle(1 + \langle v \rangle ^\gamma)} \; dv \le C_\mu,$$ 
with $\gamma>2$ and for some constant $C_\mu$ independent of $K$.  
A similar estimate holds for the last integral in \eqref{boundA2-mg-01}.   
Now by using the assumption \eqref{mg-unstab-in} and the fact that $\nu(\langle v \rangle)$ is even in $v_\theta$, we have
$$
\begin{aligned}  IIA &= 
 -  2\int _\DD  \Big(\int \langle v \rangle^{-1}Kp^- \mu_p^{-}(\langle v \rangle , Kp^-) \; dv \Big)  |\psi_*|^2\; dx \\& \le -  2c_0 K^2 \int _\DD \Big(\int r^2\langle v \rangle^{-1}(v_\theta - \psi^{(K),0})^2   \nu(\langle v \rangle) \; dv \Big)  |\psi_*|^2\; dx
 \\
 & = -  2c_0 K^2 \int _\DD \Big(\int \langle v \rangle^{-1}v^2_\theta  \nu(\langle v \rangle) \; dv \Big)   r^2 |\psi_*|^2\; dx -  2 c_0 K^2 \Big(\int \langle v \rangle^{-1}\nu(\langle v \rangle) \; dv \Big)  \int _\DD r^2|\psi^{(K),0}|^2 |\psi_*|^2\; dx
 \\&\le -  2c_0 K^2 \Big(\int \langle v \rangle^{-1}v^2_\theta  \nu(\langle v \rangle) \; dv \Big)  \| r \psi_*\|^2_{L^2(\DD)} 
 = -c_1 K^2  \| r \psi_*\|^2_{L^2(\DD)} ,   
 \end{aligned}$$
 where $c_1>0$ is independent of $K$.  
 Combining these estimates, we have therefore obtained
$$
\begin{aligned}
 \langle \mathcal{A}_2^0 \psi_*,\psi_*\rangle_{L^2}  &\le 1 -  c_1 K^2  \| r \psi_*\|^2_{L^2(\DD)}  
 + C_0C_\mu \|\psi^{(K),0}\|_{L^\infty} (K+\|\psi^{(K),0}\|_{L^\infty}).
 \end{aligned}
 $$
 Furthermore, the $L^2$ norm of $r \psi_*$ is nonzero.  We claim that $\psi^{(K),0}$ is uniformly bounded 
 independently of  $K$. 
  Indeed, recalling that  $\psi^{(K),0}$ satisfies the simple elliptic equation \eqref{equilibrium-psi} and using the decay assumption \eqref{mu-cond} on $\mu^\pm$, we have 
$$ | \Delta_r \psi^{(K),0} |  \le C_\mu\int \frac{1}{1+\langle v \rangle ^{\gamma}} \; dv \le C_\mu  $$
for some constant $C_\mu$  independent of $K$. 
Thus letting $u^{(K)}(r): = \psi^{(K),0}(r) +  C_\mu r^2 /3$, we observe that $-\Delta_r u^{(K)} \le 0$ in $\DD$ 
and $u^{(K)} =  C_\mu/3$ on the boundary $\D\DD$. 
By the standard maximum principle,  $u^{(K)}$ is bounded above and consequently so is $\psi^{(K),0}$.  
In the same way they are bounded below.  This proves the claim.  
Summarizing, we conclude that  $ \langle \mathcal{A}_2^0 \psi_*,\psi_*\rangle_{L^2} $ is dominated 
for large $K$ by $IIA$ and it is therefore strictly negative. 
\end{proof}

\appendix 
\section{Equilibria}\label{sec-Aexistence}
This appendix contains (i) the proof of regularity of our equilibria and (ii) the construction of some simple examples of equilibria. 
We recall that $\vE^0 = -\partial_r \varphi^0 e_r $ and $ B^0 = \frac 1r \partial_r (r \psi^0)$ 
where $(\varphi^0,\psi^0)$ depends only on $r$ and satisfies the elliptic ODE system  \eqref{eqs-equilibria}, 
which we rewrite as 
\begin{equation}\label{equilibria-eqs} 
- \Delta \varphi^0 = h(r,\varphi^0,\psi^0), \qquad \quad -\Delta _r \psi^0 = g(r,\varphi^0,\psi^0),\end{equation} 
$$\begin{aligned}
h(r,\varphi^0,\psi^0) :&=\int  \Big[\mu^+ \Big(\langle v \rangle  
+ \varphi^0, r (v_\theta + \psi^0)\Big) - \mu^- \Big(\langle v \rangle  - \varphi^0, r (v_\theta - \psi^0)\Big) \Big]\;dv
 \\
g(r,\varphi^0,\psi^0):&=  \int \hat v_\theta \Big[\mu^+ \Big(\langle v \rangle  
+ \varphi^0, r (v_\theta + \psi^0)\Big) - \mu^- \Big(\langle v \rangle  - \varphi^0, r (v_\theta - \psi^0)\Big) \Big] \; dv.
\end{aligned}
$$

For the regularity (i), we will verify that $\varphi^0,\psi^0 \in C(\overline \DD)$ implies that $\vE^0,B^0\in C^1(\overline \DD)$. 
Observe that $\varphi^0 _1= 1$ and $\varphi^0_2=\log r$ are two independent solutions of the 
homogeneous ODE $\Delta \varphi^0=0$ with wronskian $1/r$.   
Similarly, $\psi^0=r$ and $\varphi^0 = 1/r$ are two solutions of the homogeneous 
ODE $\Delta_r \psi^0=0$ with  wronskian  $-{2}/{r}$. 
Thus all the solutions $(\varphi^0,\psi^0)$ to \eqref{equilibria-eqs} satisfy the integral equations 
\begin{equation}\label{int-equilibria}\begin{aligned}
\varphi^0 (r)\quad&=\quad \alpha  + \int_0^r s (\log s - \log r)h(\cdot,\varphi^0,\psi^0)(s)\; ds  + \gamma \log r, 
\\\psi^0 (r)\quad&=\quad \beta r  +\frac {1}{2r}  \int_0^r (s^2 - r^2) g(\cdot,\varphi^0,\psi^0)(s)\; ds 
 + \frac \delta r, 
\end{aligned}
\end{equation}
with arbitrary constants $\alpha,\beta,\gamma,\delta$. Since $\varphi^0$ and $\psi^0$ are assumed to be 
continuous at the origin, we require $\gamma = \delta =0$.  
Clearly $h$ is continuous in $\overline \DD$, so that $\varphi^0 \in C^2 (\overline \DD)$ by \eqref{int-equilibria} 
and $\vE^0 = \nabla \varphi^0 \in C^1(\overline \DD)$.  As for $g$, we note that 
$\lim_{r \to 0^+} g(r,\varphi^0(r),\psi^0(r))=0$, which follows from the fact the integrand is odd in $v_\theta$. 
So $g$ is also continuous in all of $\overline \DD$.  
Hence, $B^0 = \frac 1r \D_r (r\psi^0) \in C^1(\overline \DD)$, as can be seen from \eqref{int-equilibria}. 
\bigskip

As for (ii), the construction of some equilibria, for simplicity we only consider the case 
 $ \mu^+(e,p) = \mu^-(e,p) = \mu(e,p)$, 
 which we take to be an arbitrary function subject to the conditions in \eqref{mu-cond}.  
Of course, $\vE^0=0$, $\vB^0=0$ is automatically an equilibrium for any $\mu$.  
However, let us consider the inhomogeneous case.  
Clearly, $h(r,0,0) = g(r,0,0) =0$.   No boundary condition is required on $\varphi^0, \psi^0$.  
It is easy to choose $\mu$ so that the functions $h(r,\cdot,\cdot) $ and $g(r,\cdot,\cdot)$ are uniformly bounded, 
and so that for all $\xi_1,\eta_1,\xi_2,\eta_2$ they satisfy 
\begin{equation}\label{assumption-gh} 
| h(r,\xi_1,\eta_1) - h(r,\xi_2,\eta_2)| + |g(r,\xi_1,\eta_1) - g(r,\xi_2,\eta_2)| \le \theta\ (|\xi_1 - \xi_2| + |\eta_1 - \eta_2|) 
\end{equation}
for some $\theta <1$. Such an assumption is satisfied for instance if $\mu$ is uniformly Lipschitz in its variables and $\mu$ replaced by $\epsilon \mu$ for sufficiently small $\epsilon$.  

Now we denote by $\mathcal{T}(\varphi^0,\psi^0)$ the right sides of the integral equations in \eqref{int-equilibria} with $\gamma=\delta =0$.   
It is clear from assumption \eqref{assumption-gh} that $\mathcal{T}$ is well-defined from 
$C([0,1]) \times C([0,1])$ into itself.  In addition, $\mathcal{T}$ is a contraction map 
on a small ball $\mathcal B$ in this space if $\alpha$ and $\beta$ are sufficiently small.  
So for each small $\alpha$ and $\beta$ there exists a unique solution  $(\varphi^0,\psi^0)$  in $\mathcal B$
to \eqref{int-equilibria} and thus to \eqref{equilibria-eqs}.


\begin{thebibliography}{99}


\bibitem{BF} {\sc J. Batt and K. Fabian,} {\em Stationary solutions of the relativistic Vlasov-Maxwell system of
plasma physics.} Chinese Ann. Math. Ser. B 14 (1993), no. 3, 253--278.


\bibitem{DiL} {\sc R. J. DiPerna and  P.-L. Lions,} {\em  Global weak solutions of Vlasov-Maxwell systems.}
 Comm. Pure Appl. Math.  42  (1989),  no. 6, 729--757.


\bibitem{GSch2.5D} {\sc R. T. Glassey and  J. Schaeffer,} {\em The "two and one-half-dimensional'' relativistic Vlasov Maxwell
 system.}
 Comm. Math. Phys.  185  (1997),  no. 2, 257--284.

\bibitem{GSch2D} {\sc R. T. Glassey and  J. Schaeffer,} 
{\em The relativistic Vlasov-Maxwell system in two space dimensions. I,
 II.}  Arch. Rational Mech. Anal.  141  (1998),  no. 4, 331--354, 355--374.


\bibitem{Guo1993} {\sc Y. Guo,} {\em Global weak solutions of the Vlasov-Maxwell system with boundary
 conditions.}
 Comm. Math. Phys.  154  (1993),  no. 2, 245--263.


\bibitem{GuoCPAM97} {\sc Y. Guo,}
{\em Stable magnetic equilibria in collisionless plasmas.}
 Comm. Pure Appl. Math.  50  (1997),  no. 9, 891--933.


\bibitem{GS95a} {\sc Y. Guo and W. A. Strauss,} {\em Instability of periodic BGK equilibria.}
 Comm. Pure Appl. Math.  48  (1995),  no. 8, 861--894.

\bibitem{GS95b} {\sc Y. Guo and W. A. Strauss,} {\em Nonlinear instability of double-humped equilibria.}
 Ann. Inst. H. Poincar\'e Anal. Non LinŽaire  12  (1995),  no. 3, 339--352.

\bibitem{GS98} {\sc Y. Guo and W. A. Strauss,} {\em Unstable BGK solitary waves and collisionless shocks.}
 Comm. Math. Phys.  195  (1998),  no. 2, 267--293.

\bibitem{GS99} {\sc Y. Guo and W. A. Strauss,} {\em Relativistic unstable periodic BGK waves.}
 Comput. Appl. Math.  18  (1999),  no. 1, 87--122.

\bibitem{GS2000} {\sc Y. Guo and W. A. Strauss,} {\em Magnetically created instability in a collisionless plasma.}
 J. Math. Pures Appl. (9)  79  (2000),  no. 10, 975--1009.

\bibitem{Hwang-Vel}  {\sc H. J. Hwang and J. Vel\'azquez,} {\em Global existence for the Vlasov-Poisson system in bounded domains.}
 Arch. Ration. Mech. Anal.  195  (2010),  no. 3, 763--796.

\bibitem{Lin} {\sc Z. Lin}, {\em Instability of periodic BGK waves.}
 Math. Res. Lett.  8  (2001),  no. 4, 521--534.
 
 \bibitem{Lin2} {\sc Z. Lin}, {\em Nonlinear instability of periodic BGK waves for the Vlasov-Poisson system.}
 Comm. Pure Appl. Math.  58 (2005), 505-528.

\bibitem{LS1} {\sc Z. Lin and W. A. Strauss}, {\em Linear Stability and instability
of relativistic Vlasov-Maxwell systems}, Comm. Pure. Appl. Math., 60 (2007), no. 5, 724--787.

\bibitem{LS2}  {\sc Z. Lin and W. A. Strauss}, {\em Nonlinear stability and instability of relativistic VlasovÐMaxwell
systems.} Commun. Pure. Appl. Math. 60, 789--837 (2007)

\bibitem{LS3} {\sc Z. Lin and W. A. Strauss}, {\em A sharp stability criterion for the Vlasov-Maxwell system}, Invent. Math. 173 (2008), no. 3, 497--546

\bibitem{PerthLions} {\sc P.L. Lions and B. Perthame,} {\em Propagation of moments and regularity for the 3-dimensional Vlasov-Poisson system,} Invent. Math., 105:415--430, 1991.

\bibitem{Pen}{\sc O. Penrose},  {\em Electrostatic instability of a non-Maxwellian plasma,} Phys. Fluids 3, 1960, pp. 258--265.

\bibitem{Pfaff} {\sc K. Pfaffelmoser,} {\em Global classical solutions of the Vlasov-Poisson system in three dimensions for general initial data,} J. Diff. Eqns., 95:281--303, 1992.

\end{thebibliography}
\end{document}